\documentclass[hidelinks,onefignum,onetabnum]{siamart220329}

\usepackage{graphicx,subcaption}
\usepackage[left=2.2cm,right=2.5cm,top=2.5cm,bottom=3.5cm]{geometry}
\usepackage{boldline}
\usepackage{tikz}
\usepackage{stmaryrd} 
\usepackage{fancyhdr}
\usepackage{titlesec}
\usepackage{secdot}
\usepackage{amsfonts}
\usepackage{epstopdf}
\usepackage{algorithmic}
\ifpdf
  \DeclareGraphicsExtensions{.eps,.pdf,.png,.jpg}
\else
  \DeclareGraphicsExtensions{.eps}
\fi

\newcommand{\llkh}{ \{\!\!\{ }
\newcommand{\rrkh}{ \}\!\!\} }
\newcommand{\VhK}{\mathbb{V}_h^k}
\newcommand{\bxi}{\boldsymbol{\xi}}
\newcommand{\ExampleHeatEqn}{Example 5.1}
\newcommand{\ExampleAnisoDiff}{Example 5.2}
\newcommand{\ExamplePorousManuf}{Example 5.3}
\newcommand{\ExamplePorousSmooth}{Example 5.4}
\newcommand{\ExamplePorousDiscontin}{Example 5.5}
\newcommand{\ExampleBlwp}{Example 5.6}

\newsiamremark{remark}{Remark}
\newsiamremark{hypothesis}{Hypothesis}
\crefname{hypothesis}{Hypothesis}{Hypotheses}
\newsiamthm{claim}{Claim}

\headers{A generalized framework for DDG methods}{M. E. Danis and J. Yan}

\title{A generalized framework for direct discontinuous Galerkin methods for nonlinear diffusion equations}

\author{Mustafa E. Danis\footnotemark[2] \thanks{Department of Aerospace Engineering, Iowa State University, Ames, 50011, USA (\email{danis@iastate.edu}).}
\thanks{asd} \and Jue Yan\thanks{Department of Mathematics, Iowa State University, Ames, 50011, USA (\email{jyan@iastate.edu}).}}
\usepackage{amsopn}

\ifpdf
\hypersetup{
  pdftitle={A generalized framework for direct discontinuous Galerkin methods for nonlinear diffusion equations},
  pdfauthor={Mustafa E. Danis, Jue Yan}
}
\fi

\begin{document}

\maketitle
\begin{abstract}
In this study, we propose a unified, general framework for the direct discontinuous Galerkin methods. In the new framework, the antiderivative of the nonlinear diffusion matrix is not needed. This allows a simple definition of the numerical flux, which can be used for general diffusion equations with no further modification. We also present the nonlinear stability analyses of the new direct discontinuous Galerkin methods and perform several numerical experiments to evaluate their performance. The numerical tests show that the symmetric and the interface correction versions of the method achieve optimal convergence and are superior to the nonsymmetric version, which demonstrates optimal convergence only for problems with diagonal diffusion matrices but loses order for even degree polynomials with a non-diagonal diffusion matrix. Singular or blow up solutions are also well captured with the new direct discontinuous Galerkin methods.
\end{abstract}



\section{Introduction}\label{sec:introduction}

In this paper, we continue to study direct discontinuous Galerkin method \cite{liuyan2008ddg} and 
other three versions of the direct discontinuous Galerkin (DDG) method \cite{liuyan2010ddgic,vidden2013sddg,yan2013} for solving the nonlinear diffusion equation
\begin{equation}\label{eqn:nonlinear-diffusion-eqn}
	\frac{\partial U}{\partial t}=\nabla\cdot(A(U)\nabla U),\hspace{0.5cm} (\textbf{x},t)\in\Omega\times(0,T),
\end{equation} 
with the initial data $U_0(\textbf{x})=U(\textbf{x},0)$. Nonlinear diffusion matrix $A(U)\in\mathbb{R}^{d\times d}$ is assumed to be positive definite. We adapt $\Omega\subset\mathbb{R}^d$ to denote the computational domain. In this study, we consider a 2-dimensional setting with $d=2$. Our focus is to derive a generalized and unified DDG method for nonlinear diffusion equations (\ref{eqn:nonlinear-diffusion-eqn}) that can be easily extended and applied to system and multi dimensional cases.

The discontinuous Galerkin (DG) method was first introduced by Reed and Hill for neutron transport equations in 1973 \cite{reed1973triangular}. However, it is after a series of papers by Cockburn, Shu et al. \cite{RKDG2,RKDG3,RKDG4,RKDG5} that the DG method became the archetype of high order methods used in the scientific community. Essentially, the DG method is a finite element method, but with a discontinuous piecewise polynomial space defined for the numerical solution and test function in each element. Due to this property, the DG method has a smaller and more compact stencil compared to its continuous counterpart. Hence, the data structure required to implement DG methods is extremely local, which allows efficient parallel computing and $hp$-adaptation. 

One of the key features of the DG method is that the communication between computational elements is established through a \textit{numerical flux} defined at element interfaces. In this regard, the DG method bears a striking similarity to finite volume method, where a Riemann solver is employed to calculate the numerical flux. Therefore, the DG method enjoys the high-order polynomial approximations as a finite element method while benefiting from the characteristic decomposition of the wave propagation provided by Riemann solvers as a finite volume method. For this reason, the DG method has been successfully applied to hyperbolic problems, i.e. compressible Euler equations, in the last three decades, cf. \cite{Cockburn-book, Shu-DGreview2014,zhang-MPS-Review}. 

On the other hand, for elliptic and parabolic problems, i.e. linear/nonlinear diffusion equations, the numerical flux must involve a proper definition for the solution gradient at element interfaces. It is, in fact, the variety of this definition that leads to several DG methods such as the interior penalty (IPDG) methods \cite{arnold1982interior,wheeler1978elliptic,baker1977finite}, the nonsymmetric interior penalty (NIPG) method \cite{NIPG1, NIPG2} and the symmetric interior penalty (SIPG) method \cite{hartmann2006a, hartmann2006b, hartmann2008}. Another important group of DG methods for solving diffusion problems include the method of Bassi and Rebay (BR) and its variations \cite{BR1, BR2, BR3, BR4}; the local DG (LDG) method \cite{LDG1998,LDG2002,LDG2002b}; the method of Baumann and Oden (BO) \cite{BO1,BO2}; hybridized DG (HDG) method \cite{Cockburn2009}. Recent works include the weakly over-penalized SIPG method \cite{Brenner2008}; weak DG \cite{LIN2015346} method and ultra weak DG method \cite{ChengYingda2008AdGf,ChenAnqi2018AUDG}. For a review of these methods, we refer to \cite{arnold2002unified,Shu2016} and the references therein. Among the many DG methods mentioned, there is little discussion of nonlinear diffusion equations except Bassi and Rebay \cite{BR1} and Local DG related methods \cite{LDG1998,Cockburn2009}. 

In addition to all the efforts mentioned above to devise a numerical flux for the diffusive terms in elliptic/parabolic equations, Liu and Yan \cite{liuyan2008ddg} introduced Direct DG (DDG) method for nonlinear diffusion equations. Inspired by the exact trace formula corresponding to the solution of heat equation with a smooth initial data that contains a discontinuous point, Liu and Yan derived a simple formula for the numerical flux to compute the solution derivative at the element interface. Although DDG method is proven to converge to the exact solution optimally when measured in an energy norm, it suffers from an order loss in the $L_2$-norm when the solution space is approximated by even degree polynomials. In order to recover optimal convergence, Liu and Yan \cite{liuyan2010ddgic} developed the direct DG method with interface correction (DDGIC). In their subsequent studies, Yan and collaborators presented symmetric and nonsymmetric versions of the DDG method \cite{vidden2013sddg, yan2013}. Even though DDG methods degenerate to the IPDG method with piecewise constant and linear polynomial approximations, there exist a number of advantages with DDG methods for higher order approximations. For such advantages, 
we refer to the discussions on a third order bound preserving scheme in \cite{chen2016}, superconvergence to $\nabla U$ in \cite{Zhang-Yan-2017,Qiu-Yan-2019} and elliptic interface problems with different jump interface conditions in \cite{Huang-Yan-2020}.

Despite the aforementioned favorable features, in the previous versions of DDG methods, the numerical flux definition is based on the antiderivative of the nonlinear diffusion matrix $A(U)$. However, this antiderivative might not exist if the diffusion matrix $A(U)$ is complicated enough. Therefore, the previous DDG methods are not applicable to nonlinear equations with such diffusion matrices. One important example where the diffusion matrix $A(U)$ cannot be integrated explicitly is the energy equation of compressible Navier-Stokes equations. 
A similar difficulty arises for the interface terms involving test function. 
This problem is addressed by defining a new direction vector on element interfaces, which depends on the nonlinear diffusion matrix $A(U)$ and geometric information of the interface.
With the introduction of the nonlinear direction vector, the evaluation of the nonlinear numerical flux is greatly simplified. Interface terms can also be clearly defined with no ambiguity.  
Danis and Yan recently applied the method in \cite{danis2021new} to solve compressible Navier-Stokes equations with DDGIC method.
This treatment of a generic diffusion process opens up the possibility of a straightforward extension of all DDG versions to the complicated nonlinear diffusion equations, which motivates this study. 
 
In this paper, the concept of the nonlinear direction vector is extended to all versions of DDG method in a generalized, unified framework. The new framework does not only address the problem of calculating the antiderivative of the diffusion matrix, but also provides an easy and practical recipe for using the DDG methods for general system of conservation laws. Moreover, interface terms of all versions of DDG methods are presented within a unified format that is clean and easy to be evaluated. Nonlinear stability analyses are presented for the new DDGIC, symmetric DDG and nonsymmetric DDG methods, and we investigate their performance in several numerical experiments. Since DDG methods degenerate to the IPDG method with low order approximations as mentioned, all numerical tests are conducted with high order polynomial approximations. In the numerical tests, optimal order of accuracy are obtained for DDGIC and symmetric DDG methods over uniform triangular meshes while a slight fraction of order loss is observed for nonsymmetric DDG method with even degree polynomial approximations. It is also shown that singular or blow up phenomena can be well captured under the new DDG framework.
 
Throughout the paper, we denote the exact solution of \Cref{eqn:nonlinear-diffusion-eqn} by the uppercase $U$ and the DG solution of \Cref{eqn:nonlinear-diffusion-eqn} by the lowercase $u$. The rest of the paper is organized as follows. In \Cref{sec:review-of-ddg}, we briefly review the direct DG methods. In \Cref{sec:scheme formulation}, the new methodology is described. In \Cref{sec:nonlinear_stab}, nonlinear stability analysis are presented. Implementation details of the new methods are explained and several numerical examples are presented in \Cref{sec:numerical-examples}. Finally, we draw our conclusions in \Cref{sec:conclusion}.
 
\section{A review of direct DG methods} \label{sec:review-of-ddg}

In this section, we will present a brief review of the original direct DG methods \cite{liuyan2008ddg, liuyan2010ddgic,vidden2013sddg,yan2013} and the required notation for later use. 

We consider a shape regular triangular mesh partition $\mathcal{T}_h$ of $\Omega$ such that $\overline{\Omega}= \cup_{{K\in \mathcal{T}_h}} K$. For each element $K$, we denote the diameter of the inscribed circle by $h_K$. Furthermore, we define the numerical solution space as 
\begin{equation*}
	\mathbb{V}_h^k:=\{v\in L^2(\Omega): v(x,y)\vline_K\in \mathbb{P}_k(K)\},
\end{equation*} 
where $\mathbb{P}_k(K)$ represents the space of polynomials of degree $k$ in two dimensions. Note that this solution space is discontinuous across element interfaces. For this purpose, we adopt the following notation for the interface solution jump and average 
\begin{equation*}
    \llbracket u \rrbracket=u^+-u^-, \hspace{1cm} \llkh u\rrkh=\frac{u^+ + u^-}{2},\quad \forall (x,y) \in \partial K,
\end{equation*} 
where $u^+$ and $u^-$ are the solution values calculated from the exterior and interior of the element $K$. 

Next, for a test function $v\in\VhK$, we multiply \Cref{eqn:nonlinear-diffusion-eqn} by $v$, integrate it by parts and apply the divergence theorem to obtain the weak form of \Cref{eqn:nonlinear-diffusion-eqn}. Along with the initial projection, the weak form is then given by
\begin{align}
    &\int_{K}u_t v ~dxdy + \int_{K} A(u)\nabla u \cdot\nabla v ~dxdy 
        -\int_{\partial K}  \widehat{A(u)\nabla u} \cdot \mathbf{n}\, v ~ds=0,\quad \forall v\in\VhK, \label{eqn:weak-form-nonlinear-diffusion-eqn}\\
    &\int_{K}u(x,y,0)v(x,y)~dxdy = \int_{K} U_0(x,y)v(x,y)~dxdy,\quad \forall v\in\VhK. \label{eqn:initial-projection}     
\end{align}
In \Cref{eqn:weak-form-nonlinear-diffusion-eqn}, the volume integration is performed over individual elements $K$ and the surface integral is performed over the element boundary $\partial K$. Here, $\mathbf{n}$ is the outward unit normal vector on $\partial K$. Furthermore, since $u(x,y,t)|_K$ is a discontinuous across the elements, $A(u)\nabla u$ is multi-valued on $\partial K$. For this reason, $A(u)\nabla u$ is written with a hat in the surface integral term in \Cref{eqn:weak-form-nonlinear-diffusion-eqn}. In fact, $\widehat{A(u)\nabla u}$ is known as the numerical flux. The original DDG method \cite{liuyan2008ddg} defines the numerical flux as
\begin{equation}\label{eqn:original-ddg-numerical-flux}
	\widehat{a_{ij}(u)u_{x_j}} = \frac{\beta_0}{h_e} \llbracket b_{ij}(u) \rrbracket n_j + \llkh b_{ij}(u)_{x_j}\rrkh + \beta_1 h_e \llbracket b_{ij}(u)_{x_1 x_j} n_1 + b_{ij}(u)_{x_2 x_j} n_2\rrbracket,
\end{equation}
where we denote by $a_{ij}(u)$ the $ij$ component of the diffusion matrix $A(u)$. Here, $b_{ij}(u)$ are the components of the matrix $B(u)$. Basically, the components of $B(u)$ are the antiderivatives of $a_{ij}(u)$ and it is defined as $b_{ij}(u)=\int^u a_{ij}(s) ds$. In \Cref{eqn:original-ddg-numerical-flux}, $h_e$ is the average of the element diameters sharing the edge $\partial K$, $n_j$ are the components of the unit normal $\mathbf{n}$ for $j=1,2$, and the subscripts $x_j$ denote the partial derivative with respect to the corresponding spatial coordinate axis for $j=1,2$.
Furthermore, $(\beta_0,\beta_1)$ is a pair of coefficients that affects the stability and optimal convergence of the DDG method. Along with \Cref{eqn:weak-form-nonlinear-diffusion-eqn,eqn:original-ddg-numerical-flux}, the definition of the original direct DG method \cite{liuyan2008ddg} is now completed.

It is well-known that the original DDG method loses an order for even degree polynomials \cite{liuyan2008ddg}. This problem is fixed either by including a jump term for the test function or introducing a \textit{numerical flux} for the test function. What determines the name of the corresponding DDG version is in fact how these additional terms are implemented. 

\subsection{The DDG method with interface correction} 

The scheme formulation of the original DDGIC method \cite{liuyan2010ddgic} is given as
\begin{equation*}\label{eqn:weak-form-ddgic}
    \begin{aligned}
        \int_{K}u_t v ~dxdy & + \int_{K} A(u)\nabla u \cdot\nabla v ~dxdy \\
        &-\int_{\partial K}  \widehat{A(u)\nabla u} \cdot \mathbf{n}\, v ~ds +\int_{\partial K}  \llbracket B(u)\rrbracket\llkh\nabla v \rrkh \cdot \mathbf{n} ~ds=0,\quad \forall v\in\VhK,   
    \end{aligned}
\end{equation*}
where the numerical flux $\widehat{A(u)\nabla u}$ is calculated using \Cref{eqn:original-ddg-numerical-flux}. Note that the test function $v$ is only nonzero inside the element $K$ by definition, thus the interface correction term is calculated as
$$
\llbracket B(u)\rrbracket\llkh\nabla v \rrkh = \frac{1}{2} \llbracket B(u)\rrbracket (\nabla v)^- .
$$
\subsection{The symmetric DDG method}

The scheme formulation of the original symmetric DDG method \cite{vidden2013sddg} is given by
\begin{equation*}\label{eqn:weak-form-symddg}
    \begin{aligned}
        \int_{K}u_t v ~dxdy &+ \int_{K} A(u)\nabla u \cdot\nabla v ~dxdy \\
        &-\int_{\partial K}  \widehat{A(u)\nabla u} \cdot \mathbf{n}\, v ~ds+\int_{\partial K} \llbracket B(u)\rrbracket\widehat{\nabla v} \cdot \mathbf{n} ~ds=0,\quad \forall v\in\VhK.
    \end{aligned}   
\end{equation*}
As in the DDGIC method, the numerical flux $\widehat{A(u)\nabla u}$ is calculated by \Cref{eqn:original-ddg-numerical-flux} while the numerical for the test function is given as
\begin{equation}\label{eqn:symmetric-ddg-test-function-numerical-flux}
    \begin{aligned}
		\widehat{v}_x&=\beta_0\frac{\llbracket v\rrbracket}{h_e} n_1+\llkh v_x\rrkh+\beta_1 h_e \llbracket v_{xx}n_1+v_{yx}n_2 \rrbracket, \\
		\widehat{v}_y&=\beta_0\frac{\llbracket v \rrbracket}{h_e} n_2+\llkh v_y\rrkh+\beta_1 h_e \llbracket v_{xy}n_1+v_{yy}n_2\rrbracket.
	\end{aligned}
\end{equation}
Note that the test function is zero outside of the element $K$. Thus, \Cref{eqn:symmetric-ddg-test-function-numerical-flux} can be simplified as
\begin{equation*}\label{eqn:symmetric-ddg-test-function-numerical-flux-simplified}
    \begin{aligned}
		\widehat{v}_x&=-\beta_0\frac{v^-}{h_e} n_1+\frac{1}{2} v_x^--\beta_1 h_e \left(v_{xx}^-n_1+v_{yx}^-n_2\right), \\
		\widehat{v}_y&=-\beta_0\frac{ v^- }{h_e} n_2+\frac{1}{2}v^-_y-\beta_1 h_e \left(v^-_{xy}n_1+v^-_{yy}n_2\right).
	\end{aligned}
\end{equation*}

\subsection{The nonsymmetric DDG method}

The scheme formulation of the original nonsymmetric DDG method \cite{yan2013} is given by
\begin{equation*}\label{eqn:weak-form-nonddg}
    \begin{aligned}
        \int_{K}u_t v ~dxdy &+ \int_{K} A(u)\nabla u \cdot\nabla v ~dxdy \\
        &-\int_{\partial K}  \widehat{A(u)\nabla u} \cdot \mathbf{n}\, v ~ds-\int_{\partial K} \widetilde{A(v)\nabla v} \cdot \mathbf{n}\llbracket u \rrbracket ~ds=0,\quad \forall v\in\VhK. 
    \end{aligned}  
\end{equation*}
Similar to the other DDG versions, the numerical flux $\widehat{A(u)\nabla u}$ is calculated by \Cref{eqn:original-ddg-numerical-flux}. The numerical flux for the test function is defined similarly but with a different penalty coefficient $\beta_{0v}$:
\begin{equation}\label{eqn:nonsymmetric-ddg-test-function-numerical-flux}
	\widetilde{a_{ij}(v)v_{x_j}} = \frac{\beta_{0v}}{h_e} \llbracket b_{ij}(v) \rrbracket n_j + \llkh b_{ij}(v)_{x_j}\rrkh + \beta_1 h_e \llbracket b_{ij}(v)_{x_1 x_j} n_1 + b_{ij}(v)_{x_2 x_j} n_2\rrbracket.
\end{equation}
Since the test function is undefined outside of the element $K$, \Cref{eqn:nonsymmetric-ddg-test-function-numerical-flux} can be simplified as
\begin{equation*}\label{eqn:nonsymmetric-ddg-test-function-numerical-flux-simplified}
	\widetilde{a_{ij}(v)v_{x_j}} = -\frac{\beta_{0v}}{h_e} b_{ij}(v^-) n_j + \frac{1}{2} b_{ij}(v^-)_{x_j} - \beta_1 h_e \left(b_{ij}(v^-)_{x_1 x_j} n_1 + b_{ij}(v^-)_{x_2 x_j} n_2)\right).
\end{equation*}

\section{The new DDG framework for nonlinear diffusion equations}\label{sec:scheme formulation}

As can be seen in the previous section, the numerical flux definition of original DDG versions is based on calculating an antiderivative matrix $B(u)$ that is calculated according to 
$$
B(u)=\int^u A(s) ds.
$$
A major drawback occurs when the components of the diffusion matrix $A(u)$ cannot be integrated explicitly. In such cases, none of the original DDG versions can be implemented. A striking example is the energy equation of compressible Navier-Stokes equations. This drawback limits the use of the original DDG versions only to simple applications where the antiderivative matrix $B(u)$ is available.

The new framework is based on the adjoint-property of inner product, which was used in the proof of a bound-preserving limiter with DDGIC method \cite{chen2016}. On the continuous level, the integrand of the surface integral in the weak form \Cref{eqn:weak-form-nonlinear-diffusion-eqn} can be rewritten as
\begin{equation*}
    A(u)\nabla u\cdot \mathbf{n} = \nabla u \cdot A(u)^T\mathbf{n}.
\end{equation*}
 By applying the adjoint-property, we define a new direction vector $\bxi(u)=A(u)^T\mathbf{n}$. The new direction vector is simply obtained by stretching/compressing and rotating the unit normal vector $\mathbf{n}$ through diffusion matrix $A(u)$. On the discrete level, the new direction vector can be calculated by
\begin{equation}\label{eqn:discrete-new-direction-vector}
    \bxi\left(\llkh u \rrkh\right)=A(\llkh u \rrkh)^T\mathbf{n}.
\end{equation}
The numerical flux can suitably be defined as
\begin{equation*}
    \widehat{A(u)\nabla u}\cdot \mathbf{n} = \widehat{\nabla u} \cdot \bxi\left(\llkh u \rrkh\right),
\end{equation*}
where the numerical flux $\widehat{\nabla u}=(\widehat{u}_x,\widehat{u}_y)$ can be computed by the original DDG numerical flux formula for the heat equation \cite{liuyan2008ddg}:
\begin{equation}\label{eqn:ddg-numflux}
	\begin{aligned}
		\widehat{u}_x&=\beta_0\frac{\llbracket u\rrbracket}{h_e} n_1+\llkh u_x\rrkh+\beta_1 h_e \llbracket u_{xx}n_1+u_{yx}n_2 \rrbracket, \\
		\widehat{u}_y&=\beta_0\frac{\llbracket u \rrbracket}{h_e} n_2+\llkh u_y\rrkh+\beta_1 h_e \llbracket u_{xy}n_1+u_{yy}n_2\rrbracket.
	\end{aligned}
\end{equation}

Now, we reformulate all DDG versions for \Cref{eqn:nonlinear-diffusion-eqn} according to the new framework: Find $u\in\VhK$ such that \begin{equation}\label{eqn:new-ddg-scheme}
    \begin{aligned}
        \int_{K}u_t v ~dxdy &+ \int_{K} A(u)\nabla u \cdot\nabla v ~dxdy \\ 
        &-\int_{\partial K}  \widehat{\nabla u} \cdot \bxi(\llkh u \rrkh) v ~ds+\sigma\int_{\partial K} \llbracket u \rrbracket \widetilde{\nabla v} \cdot \bxi(\llkh u \rrkh) ~ds=0, \quad \forall v\in\VhK,
    \end{aligned}
\end{equation}
where $\sigma=0$ for the basic DDG scheme, $\sigma=1$ for DDGIC and symmetric DDG schemes, and $\sigma=-1$ for the nonsymmetric DDG scheme. Furthermore, we denote by $\widetilde{\nabla v}$ the numerical flux for the test function $v\in\VhK$. Along with the following definitions of the numerical flux for the test function, \Cref{eqn:new-ddg-scheme} defines the new DDG versions:

\vspace{2mm}
\noindent
\underline{\emph{The baseline DDG scheme} ($\sigma=0$)}:
\begin{equation}\label{eqn:basic-ddg-test-function-numerical-flux}
    \widetilde{\nabla v}=\mathbf{0}.
\end{equation}

\noindent
\underline{\emph{The DDGIC scheme} ($\sigma=1$)}:
\begin{equation}\label{eqn:ddgic-test-function-numerical-flux}
    \widetilde{\nabla v} = \llkh \nabla v \rrkh.
\end{equation}

\noindent
\underline{\emph{The symmetric DDG scheme} ($\sigma=1$)}:
\begin{equation}\label{eqn:sym-ddg-test-function-numerical-flux}
    \begin{aligned}
    	\widetilde{v}_x&=\beta_0\frac{\llbracket v\rrbracket}{h_e} n_1+\llkh v_x\rrkh+\beta_1 h_e \llbracket v_{xx}n_1+v_{yx}n_2 \rrbracket, \\
    	\widetilde{v}_y&=\beta_0\frac{\llbracket v \rrbracket}{h_e} n_2+\llkh v_y\rrkh+\beta_1 h_e \llbracket v_{xy}n_1+v_{yy}n_2\rrbracket.
    \end{aligned}
\end{equation}

\noindent
\underline{\emph{The nonsymmetric DDG scheme} ($\sigma=-1$)}:
\begin{equation}\label{eqn:nonsym-ddg-test-function-numerical-flux}
	\begin{aligned}
		\widetilde{v}_x&=\beta_{0v}\frac{\llbracket v\rrbracket}{h_e} n_1+\llkh v_x\rrkh+\beta_1 h_e \llbracket v_{xx}n_1+v_{yx}n_2 \rrbracket, \\
		\widetilde{v}_y&=\beta_{0v}\frac{\llbracket v \rrbracket}{h_e} n_2+\llkh v_y\rrkh+\beta_1 h_e \llbracket v_{xy}n_1+v_{yy}n_2\rrbracket.
	\end{aligned}
\end{equation}
In \cite{chen2016}, the adjoint property of inner product was only used in the proof of positivity-preserving (Theorem 3.2). All numerical tests of \cite{chen2016} involving nonlinear diffusion equations were implemented  with the original DDGIC scheme formulation \cite{liuyan2010ddgic}. 

\begin{remark}
We summarize the main features and advantages of the generalized DDG methods (\ref{eqn:new-ddg-scheme}) for nonlinear diffusion equation (\ref{eqn:nonlinear-diffusion-eqn}).
\begin{itemize}
    \item Nonlinearity of the diffusion process goes into the corresponding new direction vector defined at the cell interface that greatly simplifies the implementation of the DDG method. 
    \item Numerical flux for the solution's gradient $\nabla u$ can be approximated by the linear numerical flux formula of the original DDG. Since the solution's gradient is independent of the governing equation, this allows the code reuse for general nonlinear diffusion problems. 
    \item The nonlinear direction vectors are further applied to define the interface terms involving test function. 
\end{itemize}
\end{remark}

\section{Nonlinear stability of the new DDG methods} \label{sec:nonlinear_stab}

In this section, we will discuss the nonlinear stability theory of the DDG methods developed in \Cref{sec:scheme formulation}. The important inequalities used in the proofs of the main theorems are discussed later in \Cref{sec:appendix}. 

We say that the DDG method is stable in $L_2$ sense if
\begin{equation*}
	\int_{\Omega} u^2(x,y,T) ~dxdy \leq \int_{\Omega} U^2_0(x,y) ~dxdy, \quad\forall T\geq0.
\end{equation*}
Note that the primal weak formulation of the new DDG methods is obtained by summing \Cref{eqn:initial-projection,eqn:new-ddg-scheme} over all element $K\in\mathcal{T}_h$.
\begin{align}
    &\int_{\Omega}u_t v ~dxdy+\mathbb{B}(u,v)=0,\quad \forall v\in\VhK.\label{eqn:primal-weak-form}\\
    &\int_{\Omega}u(x,y,0)v(x,y)~dxdy = \int_{\Omega} U_0(x,y)v(x,y)~dxdy,\quad \forall v\in\VhK, \label{eqn:primal-initial-projection}
\end{align}
Here, $U_0(x,y)$ denotes the initial data and $\mathbb{B}(u,v)$ is given by
\begin{equation}\label{eqn:bilinear-form}
    \mathbb{B}(u,v)=\int_{\Omega} A(u)\nabla u \cdot\nabla v~dxdy+\sum_{e\in\mathcal{E}_h}\int_{e}  \left(\llbracket v \rrbracket\widehat{\nabla u} +\sigma\llbracket u \rrbracket\widetilde{\nabla v}\right)\cdot \bxi(\llkh u \rrkh)  ~ds=0.
\end{equation}
where $\mathcal{E}_h=\cup_{{K\in \mathcal{T}_h}}\partial K$ represents the set of all element edges.

\begin{theorem}[Stability of nonsymmetric DDG method]\label{thm:stab-nonsym-ddg}
Let the model parameter $\sigma=-1$ in the scheme formulation \Cref{eqn:new-ddg-scheme} that is equipped with the numerical flux for the gradient of the numerical solution  \Cref{eqn:ddg-numflux} and the numerical flux for the gradient of the test function \Cref{eqn:nonsym-ddg-test-function-numerical-flux}. If $\beta_0 \geq \beta_{0v}$, then we have
	\begin{equation*}
		\int_{\Omega} u^2(x,y,T) ~dxdy \leq \int_{\Omega} U^2_0(x,y) ~dxdy, \quad\forall T\geq0.
	\end{equation*}
\end{theorem}

\begin{proof} By setting $u=v$, we integrate \Cref{eqn:primal-weak-form} with respect to time over $(0,T)$.
\begin{equation}\label{eqn:thm-stab-nonsym-eq1}
    \frac{1}{2}\int_{\Omega}u^2(x,y,T) ~dxdy+\int_0^T\mathbb{B}(u,u)~dt=\frac{1}{2}\int_{\Omega}u^2(x,y,0)~dxdy,
\end{equation}
where
\begin{equation}\label{eqn:thm-stab-nonsym-eq2}
    \int_0^T\mathbb{B}(u,u)~dt=\int_0^T\int_{\Omega} A(u)\nabla u \cdot\nabla u ~dxdy\,dt+\sum_{e\in\mathcal{E}_h}\int_0^T \int_{e} \llbracket u \rrbracket (\widehat{\nabla u} -\widetilde{\nabla u})\cdot \bxi(\llkh u \rrkh)  ~ds\,dt.
\end{equation}
Since $A(u)$ is positive definite, we have $A(u)\nabla u \cdot\nabla u=\nabla u^TA(u)\nabla u>0$ and thus,
\begin{equation}\label{eqn:thm-stab-nonsym-eq3}
\int_0^T\int_{\Omega} A(u)\nabla u \cdot\nabla u ~dxdy\,dt\ge0, \quad\forall T\geq0.
\end{equation}
Furthermore, we have that 
\begin{equation*}
    \widehat{\nabla u} -\widetilde{\nabla u} = \frac{\beta_0^*}{h_e}\llbracket u \rrbracket\mathbf{n},
\end{equation*}
where $\beta_0^*=\beta_0-\beta_{0v}\ge 0$ by the assumptions of the theorem. Recalling that $\bxi(\llkh u \rrkh)=A(\llkh u \rrkh)^T\mathbf{n}$ and $A(u)$ is the positive definite, we can write
\begin{equation*}
    \llbracket u \rrbracket(\widehat{\nabla u} -\widetilde{\nabla u})\cdot \bxi(\llkh u \rrkh)=\frac{\beta_0^*}{h_e}\llbracket u \rrbracket^2 \left(\mathbf{n}\cdot A(\llkh u \rrkh)^T\mathbf{n}\right) > 0.
\end{equation*}
Thus, we have
\begin{equation}\label{eqn:thm-stab-nonsym-eq4}
    \int_0^T \int_{e} \llbracket u \rrbracket (\widehat{\nabla u} -\widetilde{\nabla u})\cdot \bxi(\llkh u \rrkh)  ~ds\,dt\ge0,\quad \forall e\in\mathcal{E}_h, \,\forall T\geq 0.
\end{equation}
Substituting \Cref{eqn:thm-stab-nonsym-eq3,eqn:thm-stab-nonsym-eq4} into \Cref{eqn:thm-stab-nonsym-eq2}, and then \Cref{eqn:thm-stab-nonsym-eq2} into \Cref{eqn:thm-stab-nonsym-eq1}, we obtain
\begin{equation*}
	\int_{\Omega} u^2(x,y,T) ~dxdy \leq \int_{\Omega} u^2(x,y,0) ~dxdy, \quad\forall T\geq0.
\end{equation*}
Finally, we apply the Schwarz inequality to 
the initial projection \Cref{eqn:primal-initial-projection} with $v=u(x,y,0)$ and obtain
\begin{equation*}
    \int_{\Omega}u^2(x,y,0)~dxdy \leq  \int_{\Omega} U_0^2(x,y)~dxdy, 
\end{equation*}
which completes the proof.
\end{proof}

\begin{theorem}[Stability of symmetric DDG method]\label{thm:stab-sym-ddg}
Assume that $A(u)$ is a positive definite matrix and there exists $\gamma,\gamma^*\in\mathbb{R}$ such that the eigenvalues $(\gamma_1,\gamma_2)$ of $A(u)$ lie between $[\gamma,\gamma^*]$ for $\forall u\in\VhK$. Furthermore, let the model parameter $\sigma=1$ in the scheme formulation \Cref{eqn:new-ddg-scheme} that is equipped with the numerical flux for the gradient of the numerical solution  \Cref{eqn:ddg-numflux} and the numerical flux for the gradient of the test function \Cref{eqn:sym-ddg-test-function-numerical-flux}. If $\beta_0 \geq C(\beta_1)k^2\left(\frac{\gamma^*}{\gamma}\right)^2 \beta_1^2$, then we have
\begin{equation*}
    \begin{aligned}
        \frac{1}{2}\int_{\Omega}u^2(x,y,T) ~dxdy + \left(1-C(\beta_1)k^2\left(\frac{\gamma^*}{\gamma}\right)^2 \frac{\beta_1^2}{\beta_0} \right)\int_0^T&\int_{\Omega} A(u)\nabla u \cdot\nabla u ~dxdy\,dt \\
        &\leq \frac{1}{2}\int_{\Omega}U^2_0(x,y) ~dxdy,
    \end{aligned}
\end{equation*}
where $C(\beta_1)=C_1/2\beta_1^2 + 2C_2>0$ and $C_1,C_2>0$ are constants.
\end{theorem}

\begin{proof}
By setting $u=v$, \Cref{eqn:primal-weak-form} becomes
\begin{equation}\label{eqn:time-derive-ddgsym-stab}
    \frac{1}{2}\frac{d}{dt}\int_{\Omega}u^2 ~dxdy+\mathbb{B}(u,u)=0,
\end{equation}
where
\begin{equation}\label{eqn:bilinear-form-sym-ddg} 
    \mathbb{B}(u,u)=\int_{\Omega} A(u)\nabla u \cdot\nabla u ~dxdy 
+2\sum_{e\in\mathcal{E}_h}\int_{e}  \llbracket u \rrbracket\widehat{\nabla u}\cdot \bxi(\llkh u \rrkh)  ~ds=0.
\end{equation}
Note that
\begin{equation*} 
    \llbracket u \rrbracket\widehat{\nabla u}\cdot \bxi(\llkh u \rrkh)=\frac{\beta_0}{h} \llbracket u \rrbracket^2 \mathbf{n}\cdot\bxi(\llkh u \rrkh)+\llbracket u \rrbracket\llkh \nabla u \rrkh\cdot\bxi(\llkh u \rrkh)+\beta_1 h\llbracket u \rrbracket\llbracket\nabla (\nabla u\cdot\mathbf{n}) \rrbracket\cdot\bxi(\llkh u \rrkh).
\end{equation*}
Therefore, invoking \Cref{lemma:surf_jump,lemma:surf_avg,lemma:surf_jump_second_deriv}, we get
\begin{equation*}
    \begin{aligned}
        \sum_{e\in\mathcal{E}_h}\int_{e}  \llbracket u \rrbracket\widehat{\nabla u}\cdot \bxi&(\llkh u \rrkh)  ~ds \\
        &\geq \sum_{e\in\mathcal{E}_h}\frac{\gamma\beta_0}{h} \left\| \llbracket u \rrbracket \right\|^2_{L^2(e)} \\
        &-\sum_{e\in\mathcal{E}_h}\frac{\gamma\beta_0}{2h}\left\| \llbracket u \rrbracket \right\|^2_{L^2(e)} -\sum_{K\in\mathcal{T}_h} C_1\frac{(\gamma^*k)^2}{4\gamma\beta_0}\left\| \nabla u) \right\|^2_{L^2(K)} \\ 
        &- \sum_{e\in\mathcal{E}_h}\frac{\gamma\beta_0}{2h}\left\| \llbracket u \rrbracket \right\|^2_{L^2(e)} 
        -\sum_{K\in\mathcal{T}_h} C_2\frac{(\gamma^*\beta_1k)^2}{\gamma\beta_0} \left\|\nabla u \right\|^2_{L^2(K)} \\
        &=-\sum_{K\in\mathcal{T}_h} C_1\frac{(\gamma^*k)^2}{4\gamma\beta_0}\left\| \nabla u) \right\|^2_{L^2(K)}-\sum_{K\in\mathcal{T}_h} C_2\frac{(\gamma^*\beta_1k)^2}{\gamma\beta_0} \left\|\nabla u \right\|^2_{L^2(K)} \\
        &=-\sum_{K\in\mathcal{T}_h} \frac{1}{2}\left(\frac{C_1}{2\beta_1^2} + 2C_2\right)\frac{(\gamma^*\beta_1k)^2}{\gamma\beta_0}\left\| \nabla u) \right\|^2_{L^2(K)}. 
    \end{aligned}
\end{equation*}
This can be rewritten as
\begin{equation*}
    \sum_{e\in\mathcal{E}_h}\int_{e}  \llbracket u \rrbracket\widehat{\nabla u}\cdot \bxi(\llkh u \rrkh)  ~ds \geq -\frac{C(\beta_1)k^2}{2}\frac{(\gamma^*)^2}{\gamma} \frac{\beta_1^2}{\beta_0}\sum_K \left\| \nabla u) \right\|^2_{L^2(K)}.
\end{equation*}
where $C(\beta_1)=C_1/2\beta_1^2 + 2C_2$. Next, we use the assumption on the eigenvalues of $A(u)$:
\begin{equation*}
    \sum_K \left\| \nabla u) \right\|^2_{L^2(K)}\leq\frac{1}{\gamma}\sum_K \int_K \nabla u\cdot A(u)\nabla u ~dxdy,
\end{equation*} 
and obtain
\begin{equation}\label{eqn:symDDG_nablaU_norm}
    \sum_{e\in\mathcal{E}_h}\int_{e}  \llbracket u \rrbracket\widehat{\nabla u}\cdot \bxi(\llkh u \rrkh)  ~ds \geq -\frac{C(\beta_1)k^2}{2}\left(\frac{\gamma^*}{\gamma}\right)^2 \frac{\beta_1^2}{\beta_0} \sum_K \int_K \nabla u\cdot A(u)\nabla u ~dxdy. 
\end{equation}
Substituting this \Cref{eqn:symDDG_nablaU_norm} into \Cref{eqn:bilinear-form-sym-ddg}, and then, \Cref{eqn:bilinear-form-sym-ddg} into \Cref{eqn:time-derive-ddgsym-stab} gives
\begin{equation} \label{eqn:sym-ddg-final-dt}
    \frac{1}{2}\frac{d}{dt}\int_{\Omega}u^2 ~dxdy + \left(1-C(\beta_1)k^2\left(\frac{\gamma^*}{\gamma}\right)^2 \frac{\beta_1^2}{\beta_0} \right)\int_{\Omega} A(u)\nabla u \cdot\nabla u ~dxdy \leq 0.
\end{equation}
Lastly, we integrate \Cref{eqn:sym-ddg-final-dt} over $(0,T)$ and recall
\begin{equation*}
    \int_{\Omega}u^2(x,y,0)~dxdy \leq  \int_{\Omega} U_0^2(x,y)~dxdy, 
\end{equation*}
from the proof of \Cref{thm:stab-nonsym-ddg}. This completes the proof provided that we have $\beta_0 \geq C(\beta_1)k^2\left(\frac{\gamma^*}{\gamma}\right)^2 \beta_1^2$.
\end{proof}

\begin{theorem}[Stability of DDGIC method]\label{thm:stab-ddgic}
Assume that $A(u)$ is a positive definite matrix and there exists $\gamma,\gamma^*\in\mathbb{R}$ such that the eigenvalues $(\gamma_1,\gamma_2)$ of $A(u)$ lie between $[\gamma,\gamma^*]$ for $\forall u\in\VhK$. Furthermore, let the model parameter $\sigma=1$ in the scheme formulation \Cref{eqn:new-ddg-scheme} that is equipped with the numerical flux for the gradient of the numerical solution  \Cref{eqn:ddg-numflux} and the numerical flux for the gradient of the test function \Cref{eqn:ddgic-test-function-numerical-flux}. If $\beta_0 \geq C(\beta_1)k^2\left(\frac{\gamma^*}{\gamma}\right)^2 \beta_1^2$, then we have
\begin{equation*}
    \begin{aligned}
        \frac{1}{2}\int_{\Omega}u^2(x,y,T) ~dxdy + \left(1-C(\beta_1)k^2\left(\frac{\gamma^*}{\gamma}\right)^2 \frac{\beta_1^2}{\beta_0} \right)\int_0^T&\int_{\Omega} A(u)\nabla u \cdot\nabla u ~dxdy\,dt \\
        &\leq \frac{1}{2}\int_{\Omega}U^2_0(x,y) ~dxdy.
    \end{aligned}
\end{equation*}
where $C(\beta_1)=C_1/\beta_1^2 + C_2>0$ and $C_1,C_2>0$ are constants.
\end{theorem}

\begin{proof}
The proof is very similar to the proof of \Cref{thm:stab-sym-ddg}. Therefore, we will only lay out the sketch of the proof. By setting $u=v$, \Cref{eqn:primal-weak-form} becomes
\begin{equation*}\label{eqn_time_derive_ddgic_stab}
    \frac{1}{2}\frac{d}{dt}\int_{\Omega}u^2 ~dxdy+\mathbb{B}(u,u)=0,
\end{equation*}
where
\begin{equation*}\label{eqn:bilinear-form-ddgic} 
    \mathbb{B}(u,u)=\int_{\Omega} A(u)\nabla u \cdot\nabla u ~dxdy + \sum_{e\in\mathcal{E}_h}\int_{e}  \llbracket u \rrbracket\left(\widehat{\nabla u}+\llkh\nabla u\rrkh\right)\cdot \bxi(\llkh u \rrkh)  ~ds=0.
\end{equation*}
Note that 
\begin{equation*}
    \begin{aligned}
        \llbracket u \rrbracket\left(\widehat{\nabla u}+\llkh\nabla u\rrkh\right)\cdot \bxi(\llkh u \rrkh)  &= \frac{\beta_0}{h} \llbracket u \rrbracket^2 \mathbf{n}\cdot\bxi(\llkh u \rrkh) \\
        &+2\llbracket u \rrbracket\llkh \nabla u \rrkh\cdot\bxi(\llkh u \rrkh)+\beta_1 h\llbracket u \rrbracket\llbracket\nabla (\nabla u\cdot\mathbf{n}) \rrbracket\cdot\bxi(\llkh u \rrkh).
    \end{aligned}
\end{equation*}
As in the proof of \Cref{thm:stab-sym-ddg}, we invoke \Cref{lemma:surf_jump,lemma:surf_avg,lemma:surf_jump_second_deriv}:
\begin{equation*}
    \begin{aligned}
        \sum_{e\in\mathcal{E}_h}\int_{e}  \llbracket u \rrbracket\left(\widehat{\nabla u}+\llkh\nabla u\rrkh\right)\cdot \bxi&(\llkh u \rrkh) ~ds \\
        &\geq \sum_{e\in\mathcal{E}_h}\frac{\gamma\beta_0}{h} \left\| \llbracket u \rrbracket \right\|^2_{L^2(e)} \\
        &-\sum_{e\in\mathcal{E}_h}\frac{\gamma\beta_0}{2h}\left\| \llbracket u \rrbracket \right\|^2_{L^2(e)} -\sum_{K\in\mathcal{T}_h} C_1\frac{(\gamma^*k)^2}{\gamma\beta_0}\left\| \nabla u) \right\|^2_{L^2(K)} \\ 
        &- \sum_{e\in\mathcal{E}_h}\frac{\gamma\beta_0}{2h}\left\| \llbracket u \rrbracket \right\|^2_{L^2(e)} 
        -\sum_{K\in\mathcal{T}_h} C_2\frac{(\gamma^*\beta_1k)^2}{\gamma\beta_0} \left\|\nabla u \right\|^2_{L^2(K)} \\
        &=-C(\beta_1)k^2\frac{(\gamma^*)^2}{\gamma} \frac{\beta_1^2}{\beta_0}\sum_K \left\| \nabla u) \right\|^2_{L^2(K)},
    \end{aligned}
\end{equation*}
where $C(\beta_1)=C_1/\beta_1^2 + C_2$. Then, following the same lines of steps as in the proof of \Cref{thm:stab-sym-ddg} will lead to the desired result.
\end{proof}

\section{Numerical Examples}\label{sec:numerical-examples}

In this section, the results of several numerical examples are presented. All results are obtained on a square domain $\Omega=[x_0,x_0+L]\times[y_0,y_0+L]$ where we denote by $x_0,y_0$ the origin of the $(x,y)$ coordinate system, and $L$ is the domain size in the x and y directions. A set of uniform triangular meshes is used to discretize the computational domain $\Omega$ as shown in \Cref{fig:mesh_set}. In all simulations, we set $\beta_0=(k+1)^2$, $\beta_{0v}=\frac{\beta_0}{2}$ and $\beta_1=\frac{1}{2k(k+1)}$ in the numerical flux definitions, and the time integration is performed by a third order explicit strong stability-preserving (SSP) Runge-Kutta scheme \cite{Shu-Osher-1988}. Unless stated otherwise, the time step size $\Delta t$ is determined by the following Courant-Friedrichs-Levy (CFL) rule:
\begin{equation}\label{eqn:cfl-condition}
    \Delta t\frac{\mu}{\min_K{h^2_K}}<\omega\lambda,
\end{equation}
where $\lambda$ is the CFL number, $\mu$ is a diffusion constant and $\omega$ is the minimum quadrature weight for the volume integration developed in \cite{ZhangXiangxiong2012MaPH}.

Furthermore, the convergence rates are reported only in the $L_2$ and $L_\infty$ norms. To do that, we employ the $(k+1)th$ order quadrature rule to calculate the $L_2$-error while the $L_\infty$-errors are measured using $361$ points generated by the same quadrature rule in each element. 

\begin{figure}[h!]
	\centering
	\begin{subfigure}[b]{0.23\textwidth}
		\centering
		\includegraphics[scale=0.17,trim={2cm 0 2cm 0},clip]{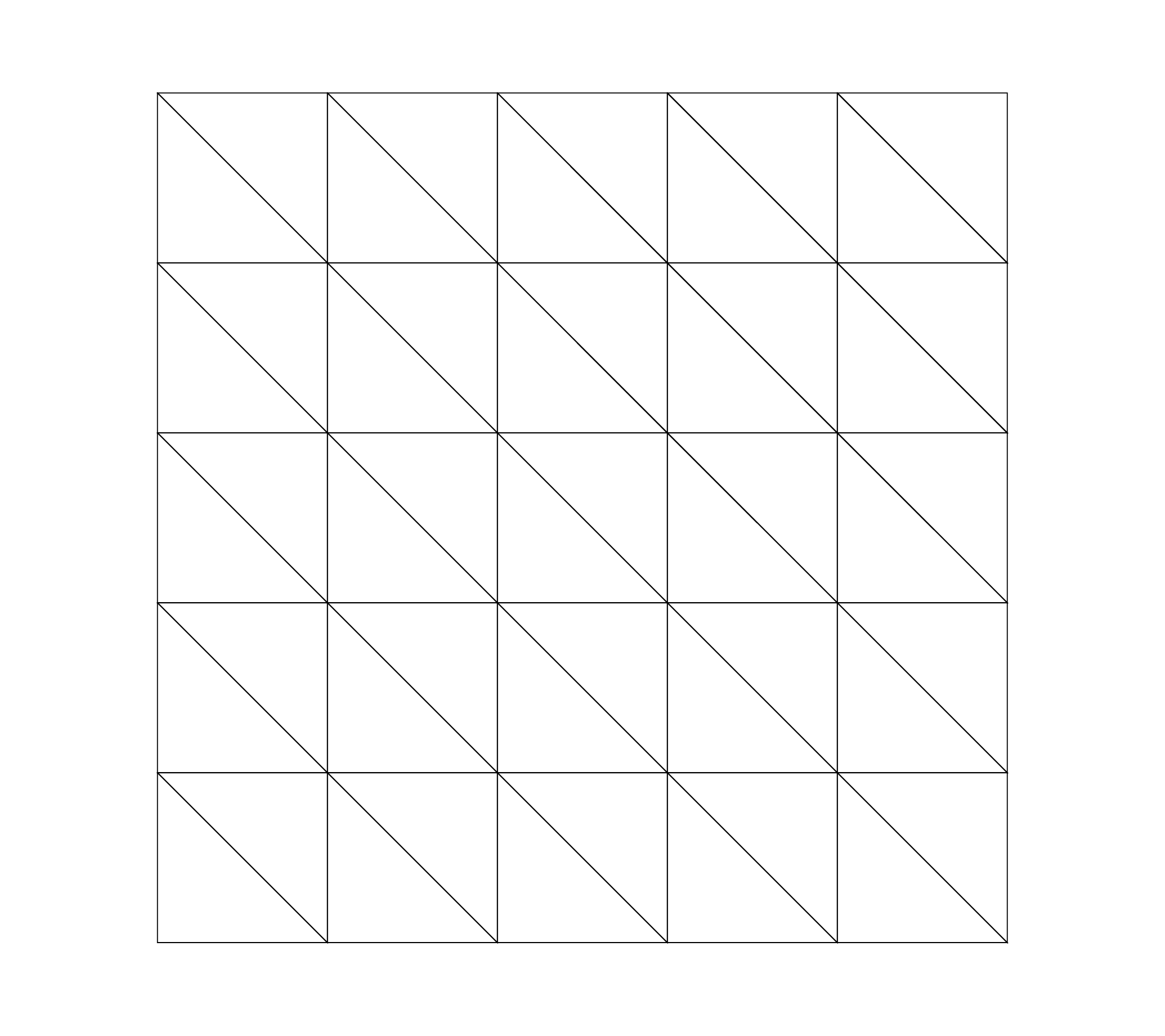}
		\caption{$h$}
	\end{subfigure}
	\begin{subfigure}[b]{0.23\textwidth}
		\centering
		\includegraphics[scale=0.17,trim={2cm 0 2cm 0},clip]{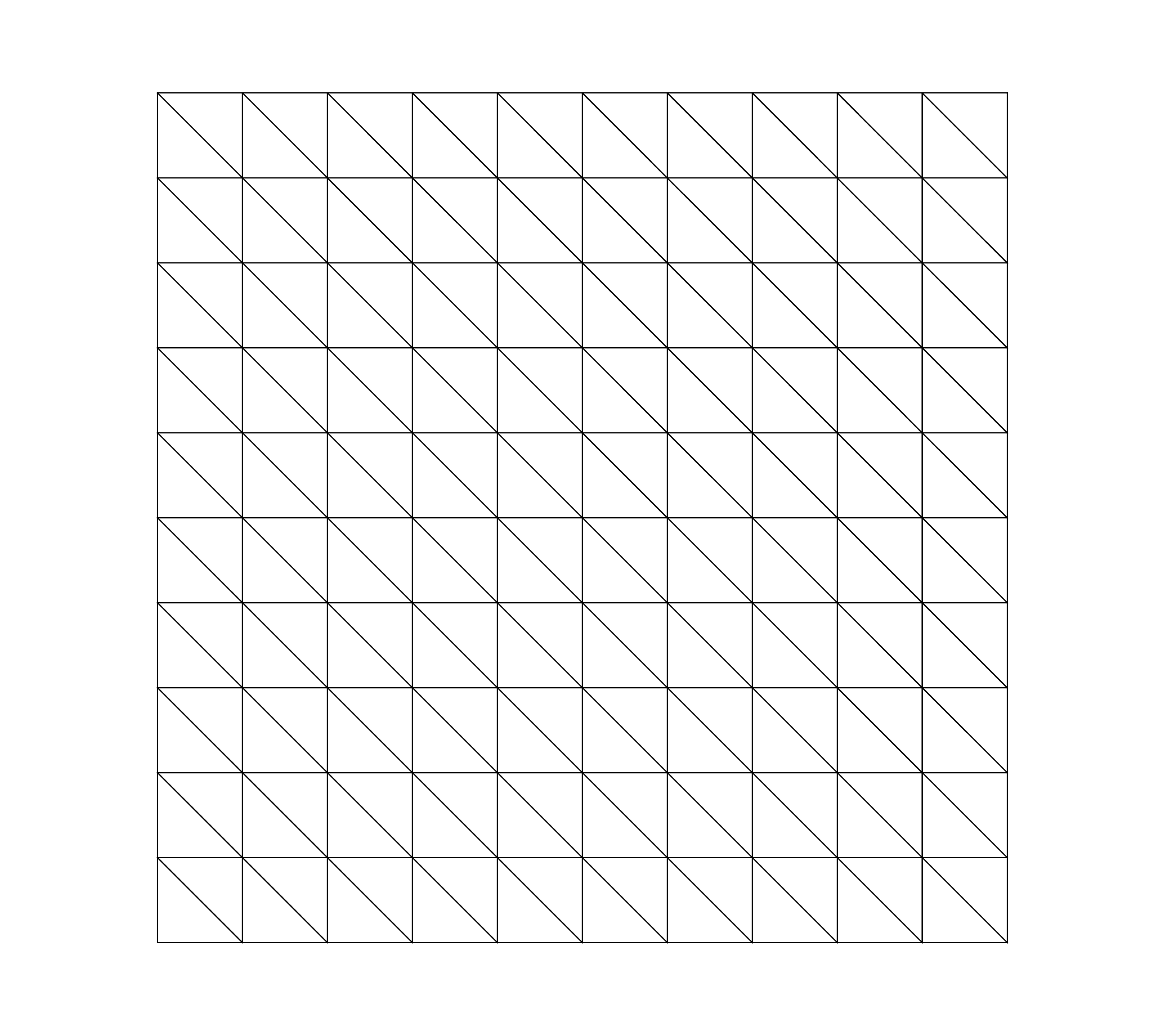}
		\caption{$h/2$}
	\end{subfigure}
	\begin{subfigure}[b]{0.23\textwidth}
		\centering
		\includegraphics[scale=0.17,trim={2cm 0 2cm 0},clip]{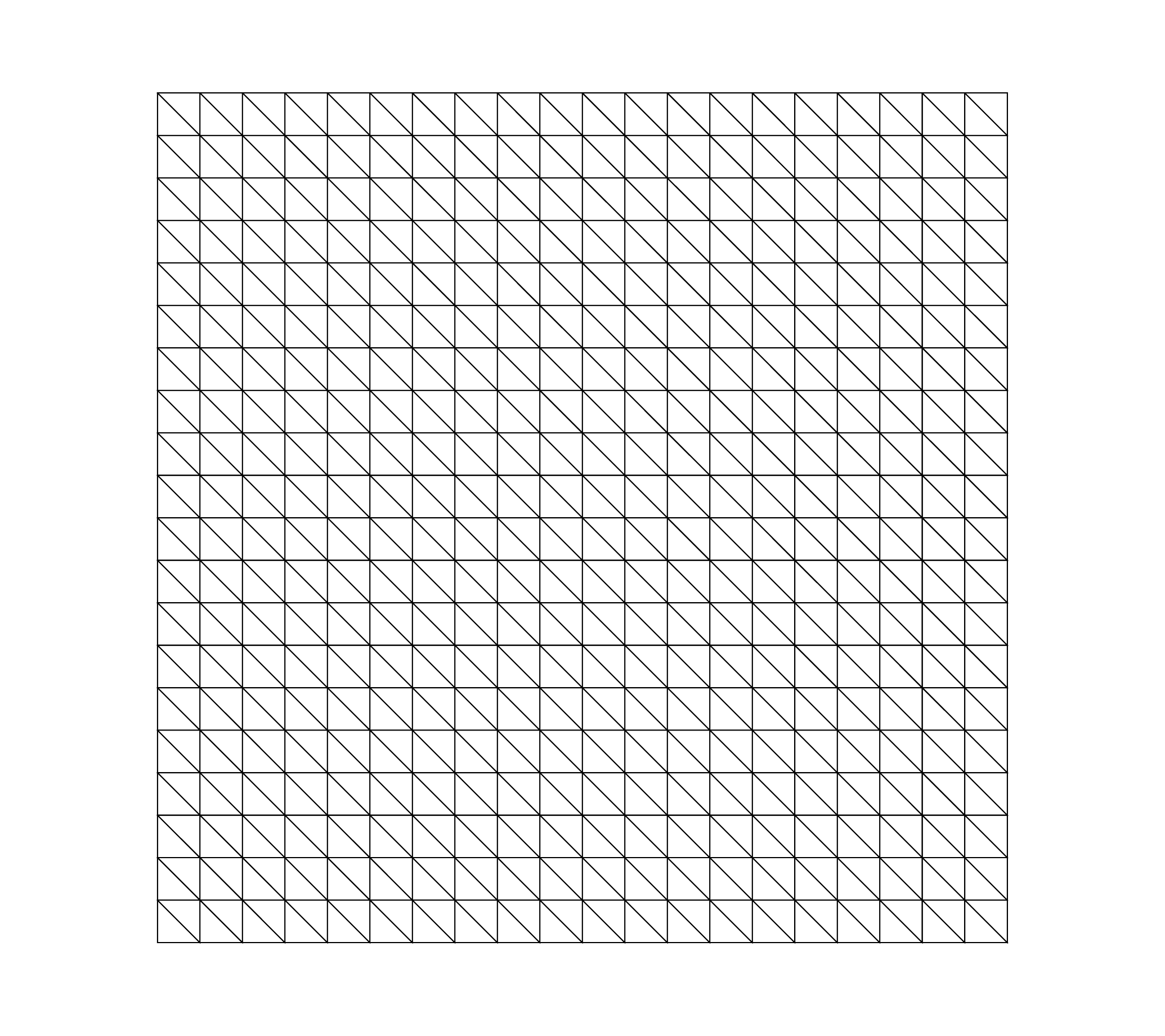}
		\caption{$h/4$}
	\end{subfigure}
	\begin{subfigure}[b]{0.23\textwidth}
		\centering
		\includegraphics[scale=0.17,trim={2cm 0 2cm 0},clip]{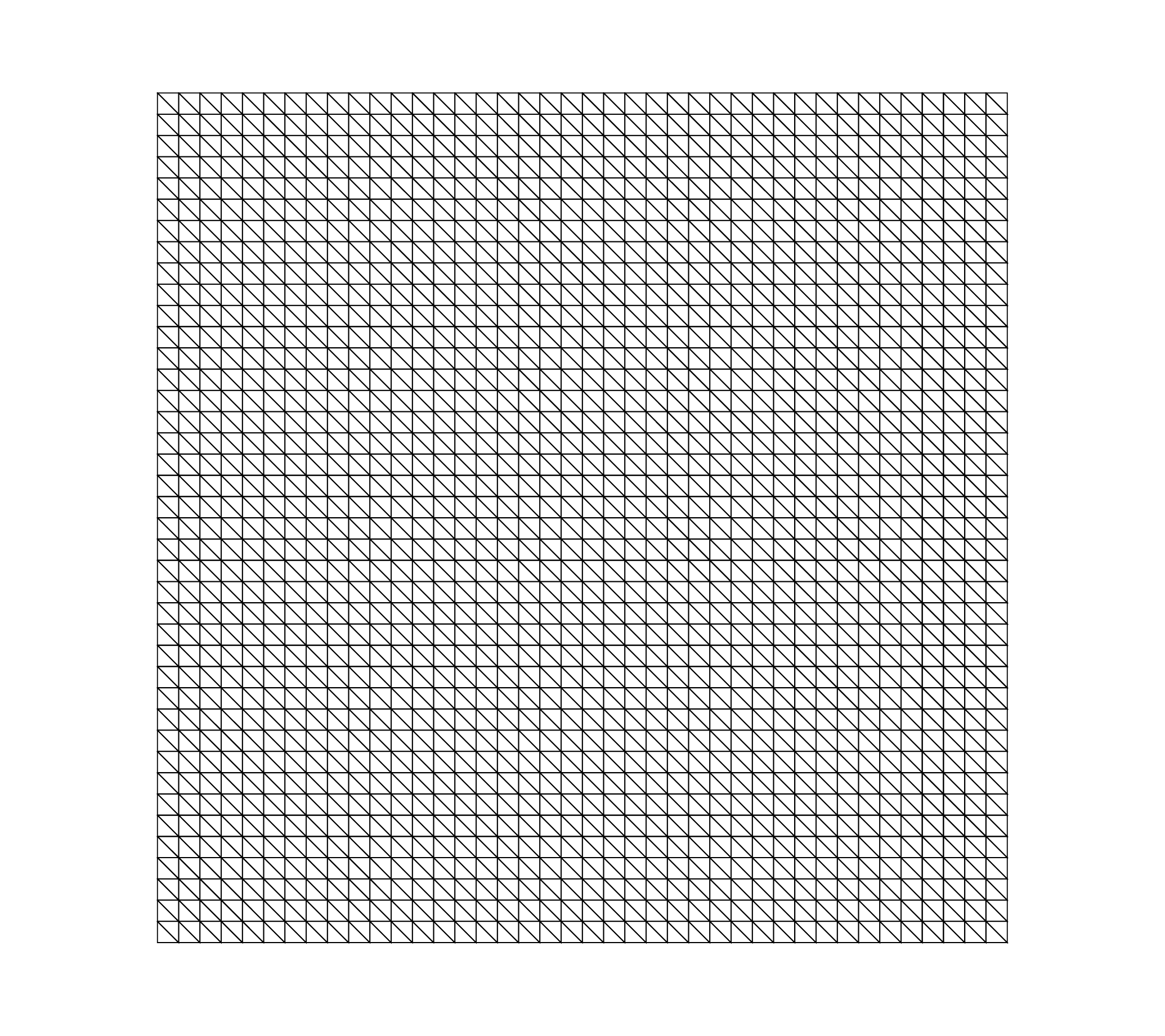}
		\caption{$h/8$}
	\end{subfigure}
	\caption{The set of uniform unstructured meshes with $h=L/5$}\label{fig:mesh_set}
\end{figure}
\vspace{0.5cm}
\hypertarget{example:heat-equation}{}
\noindent\textbf{\ExampleHeatEqn.} 
In this example, we consider the heat equation
\begin{equation*}
	\frac{\partial U}{\partial t}=\mu\Delta U,
\end{equation*}
with periodic boundary conditions on $\Omega=[0,1]\times[0,1]$ where $\mu$ is a constant diffusion coefficient. Note that the diffusion matrix $A(u)$ is given as
$$
A(u)=\mu\mathbb{I},
$$
where $\mathbb{I}\in\mathbb{R}^{2\times2}$ is the identity matrix. The initial condition for this example is obtained from the following exact solution at $t=0$: 
\begin{equation*}
    U(x,y,t)=e^{-8\pi^2\mu t}\cos\left(2\pi(x+y)\right).
\end{equation*}
In this example, we set $\lambda=0.1$, $\mu=0.01$ and $T=1$. In addition, all quadrature rules are exact up to polynomials of degree $2k+1$.

The $L_2$ and $L_\infty$ errors are listed in \Cref{table:heat-equation-L2,table:heat-equation-Linf}, respectively. We observe that optimal $(k+1)th$ order convergence in the $L_2$-norm is obtained by the DDGIC and symmetric DDG methods without any order loss for even degree polynomials. On the other hand, the nonsymmetric DDG method is optimally convergent only for $k=3$ in the $L_2$-norm. For even degree polynomials, the convergence rate of the nonsymmetric DDG method is degrading as the mesh is refined. However, in the $L_\infty$-norm, all DDG methods demonstrate $(k+1)th$ order optimal convergence.  

\begin{table}[h!]
    \centering
    \caption{$L_2$ errors for \protect\hyperlink{example:heat-equation}{\ExampleHeatEqn} at $T=1$}\label{table:heat-equation-L2}
    \begin{tabular}{cccccccc}
        \hlineB{3}
              & \multicolumn{7}{c}{$L_2$ errors and orders for \textit{the DDGIC method}}          \\ \cline{2-8} 
              & $h$      & $h/2$    & Order & $h/4$    & Order & $h/8$    & Order \\ \hline
        $k=2$ & 2.47E-03 & 3.10E-04 & 3.00  & 3.88E-05 & 3.00  & 4.85E-06 & 3.00  \\
        $k=3$ & 1.87E-04 & 1.17E-05 & 4.00  & 7.29E-07 & 4.00  & 4.55E-08 & 4.00  \\
        $k=4$ & 1.16E-05 & 3.64E-07 & 5.00  & 1.14E-08 & 5.00  & 3.55E-10 & 5.00  \\ \hlineB{2}
              &          &          &       &          &       &          &       \\ \hlineB{2}
              & \multicolumn{7}{c}{$L_2$ errors and orders for \textit{the symmetric DDG method}}      \\ \cline{2-8} 
              & $h$      & $h/2$    & Order & $h/4$    & Order & $h/8$    & Order \\ \hline
        $k=2$ & 2.82E-03 & 3.55E-04 & 2.99  & 4.45E-05 & 2.99  & 5.56E-06 & 3.00  \\
        $k=3$ & 2.10E-04 & 1.30E-05 & 4.01  & 8.10E-07 & 4.00  & 5.06E-08 & 4.00  \\
        $k=4$ & 1.28E-05 & 4.02E-07 & 4.99  & 1.26E-08 & 5.00  & 3.93E-10 & 5.00  \\ \hlineB{2}
              &          &          &       &          &       &          &       \\ \hlineB{2}
              & \multicolumn{7}{c}{$L_2$ errors and orders for \textit{the nonsymmetric DDG method}}      \\ \cline{2-8} 
              & $h$      & $h/2$    & Order & $h/4$    & Order & $h/8$    & Order \\ \hline
        $k=2$ & 2.29E-03 & 2.88E-04 & 2.99  & 3.73E-05 & 2.95  & 5.20E-06 & 2.84  \\
        $k=3$ & 1.85E-04 & 1.15E-05 & 4.01  & 7.21E-07 & 4.00  & 4.50E-08 & 4.00  \\
        $k=4$ & 1.13E-05 & 3.66E-07 & 4.95  & 1.26E-08 & 4.86  & 5.12E-10 & 4.62  \\ \hlineB{3}
    \end{tabular}    
\end{table}

\begin{table}[h!]
    \centering
    \caption{$L_\infty$ errors for \protect\hyperlink{example:heat-equation}{\ExampleHeatEqn} at $T=1$}\label{table:heat-equation-Linf}
    \begin{tabular}{cccccccc}
        \hlineB{3}
              & \multicolumn{7}{c}{$L_\infty$ errors and orders for \textit{the DDGIC method}}          \\ \cline{2-8} 
              & $h$      & $h/2$    & Order & $h/4$    & Order & $h/8$    & Order \\ \hline
        $k=2$ & 8.58E-03 & 1.03E-03 & 3.06  & 1.31E-04 & 2.97  & 1.65E-05 & 2.99  \\
        $k=3$ & 7.77E-04 & 5.40E-05 & 3.85  & 3.43E-06 & 3.98  & 2.17E-07 & 3.98  \\
        $k=4$ & 4.63E-05 & 1.35E-06 & 5.10  & 4.23E-08 & 4.99  & 1.32E-09 & 5.00  \\ \hlineB{2}
              &          &          &       &          &       &          &       \\ \hlineB{2}
              & \multicolumn{7}{c}{$L_\infty$ errors and orders for \textit{the symmetric DDG method}}      \\ \cline{2-8} 
              & $h$      & $h/2$    & Order & $h/4$    & Order & $h/8$    & Order \\ \hline
        $k=2$ & 6.15E-03 & 7.01E-04 & 3.13  & 8.96E-05 & 2.97  & 1.13E-05 & 2.99  \\
        $k=3$ & 6.18E-04 & 4.36E-05 & 3.83  & 2.83E-06 & 3.94  & 1.79E-07 & 3.98  \\
        $k=4$ & 3.39E-05 & 1.01E-06 & 5.07  & 3.26E-08 & 4.96  & 1.03E-09 & 4.99  \\ \hlineB{2}
              &          &          &       &          &       &          &       \\ \hlineB{2}
              & \multicolumn{7}{c}{$L_\infty$ errors and orders for \textit{the nonsymmetric DDG method}}      \\ \cline{2-8} 
              & $h$      & $h/2$    & Order & $h/4$    & Order & $h/8$    & Order \\ \hline
        $k=2$ & 1.14E-02 & 1.39E-03 & 3.03  & 1.83E-04 & 2.92  & 2.32E-05 & 2.98  \\
        $k=3$ & 9.58E-04 & 6.56E-05 & 3.87  & 4.10E-06 & 4.00  & 2.59E-07 & 3.98  \\
        $k=4$ & 6.04E-05 & 1.86E-06 & 5.02  & 6.08E-08 & 4.94  & 1.96E-09 & 4.96  \\ \hlineB{3}
    \end{tabular}
\end{table}

\vspace{0.5cm}
\hypertarget{example:anisotropic-linear-diffusion}{}
\noindent\textbf{\ExampleAnisoDiff.} 
In this numerical test, we now consider an anisotropic diffusion equation with mixed derivatives
\begin{equation*}
    \frac{\partial U}{\partial t}=\mu\left(2U_{xx}+3U_{xy}+3U_{yy}\right),
\end{equation*}
with periodic boundary conditions on $\Omega=[0,1]\times[0,1]$. Note that this equation is still linear, i.e. diffusion matrix $A(u)$ is a constant coefficient matrix. Moreover, it can be written in a nonsymmetric form
\begin{equation*}
    A(u)=
    \mu\begin{pmatrix}
    	2 & 1 \\
    	2 & 3
    \end{pmatrix}.
\end{equation*} 
A nonsymmetric diffusion matrix is chosen to see how the new DDG methods would behave in such a setting. The initial condition is set from the following exact solution at $t=0$:
\begin{equation}
    U(x,y,t)=e^{-32\pi^2\mu t}\cos(2\pi y)\cos(4\pi x-2\pi y).
\end{equation}
As in the previous example, we set $\lambda=0.1$, $\mu=0.01$, $T=1$, and all quadrature rules are exact up to polynomials of degree $2k+1$.   

The $L_2$ and $L_\infty$ errors are listed in \Cref{table:anisotropic-linear-diffusion-L2,table:anisotropic-linear-diffusion-Linf}, respectively. The DDGIC and symmetric DDG methods behave similarly in all cases. For the even degree polynomials, these methods do not lose order. On the other hand, the nonsymmetric DDG method demonstrates optimal $(k+1)$ convergence only for $k=3$ while it loses an order for even degree polynomials, which is clearer compared to \hyperlink{example:heat-equation}{\ExampleHeatEqn}. Furthermore, the order loss is not only seen in the $L_2$-error but also clearly observed in the $L_\infty$-error in this case.

\begin{remark}
This problem has been also tested equivalently with the diffusion matrix
$$
A(u)=
\mu\begin{pmatrix}
	2 & 1.5 \\
	1.5 & 3
\end{pmatrix},
$$
which is symmetric-positive definite. We note that the performance of the new DDG methods does not change with this diffusion matrix and the same results in \Cref{table:anisotropic-linear-diffusion-L2,table:anisotropic-linear-diffusion-Linf} are obtained. Therefore, those results are not reported.
\end{remark}


\begin{table}[h!]
    \centering
    \caption{$L_2$ errors for \protect\hyperlink{example:anisotropic-linear-diffusion}{\ExampleAnisoDiff} at $T=1$}\label{table:anisotropic-linear-diffusion-L2}
    \begin{tabular}{cccccccc}
        \hlineB{3}
              & \multicolumn{7}{c}{$L_2$ errors and orders for \textit{the DDGIC method}}          \\ \cline{2-8} 
              & $h$      & $h/2$    & Order & $h/4$    & Order & $h/8$    & Order \\ \hline
        $k=2$ & 1.34E-02 & 2.59E-03 & 2.37  & 2.14E-04 & 3.60  & 1.76E-05 & 3.61  \\
        $k=3$ & 3.37E-03 & 1.16E-04 & 4.87  & 5.31E-06 & 4.45  & 3.12E-07 & 4.09  \\
        $k=4$ & 4.62E-04 & 1.19E-05 & 5.29  & 3.73E-07 & 4.99  & 1.18E-08 & 4.99  \\ \hlineB{2}
              &          &          &       &          &       &          &       \\ \hlineB{2}
              & \multicolumn{7}{c}{$L_2$ errors and orders for \textit{the symmetric DDG method}}      \\ \cline{2-8} 
              & $h$      & $h/2$    & Order & $h/4$    & Order & $h/8$    & Order \\ \hline
        $k=2$ & 1.41E-02 & 3.01E-03 & 2.22  & 2.56E-04 & 3.56  & 1.96E-05 & 3.71  \\
        $k=3$ & 3.68E-03 & 1.27E-04 & 4.86  & 5.51E-06 & 4.53  & 3.20E-07 & 4.11  \\
        $k=4$ & 4.92E-04 & 1.22E-05 & 5.33  & 3.88E-07 & 4.98  & 1.22E-08 & 4.99  \\ \hlineB{2}
              &          &          &       &          &       &          &       \\ \hlineB{2}
              & \multicolumn{7}{c}{$L_2$ errors and orders for \textit{the nonsymmetric DDG method}}      \\ \cline{2-8} 
              & $h$      & $h/2$    & Order & $h/4$    & Order & $h/8$    & Order \\ \hline
        $k=2$ & 1.17E-02 & 2.47E-03 & 2.25  & 4.34E-04 & 2.51  & 9.21E-05 & 2.24  \\
        $k=3$ & 3.67E-03 & 2.00E-04 & 4.20  & 1.07E-05 & 4.23  & 6.35E-07 & 4.07  \\
        $k=4$ & 5.70E-04 & 1.15E-05 & 5.63  & 5.70E-07 & 4.34  & 3.47E-08 & 4.04  \\ \hlineB{3}
    \end{tabular}
\end{table}

\begin{table}[h!]
    \centering
    \caption{$L_\infty$ errors for \protect\hyperlink{example:anisotropic-linear-diffusion}{\ExampleAnisoDiff} at $T=1$}\label{table:anisotropic-linear-diffusion-Linf}
    \begin{tabular}{cccccccc}
        \hlineB{3}
              & \multicolumn{7}{c}{$L_\infty$ errors and orders for \textit{the DDGIC method}}          \\ \cline{2-8} 
              & $h$      & $h/2$    & Order & $h/4$    & Order & $h/8$    & Order \\ \hline
        $k=2$ & 2.28E-02 & 4.33E-03 & 2.40  & 4.03E-04 & 3.43  & 4.04E-05 & 3.32  \\
        $k=3$ & 7.81E-03 & 3.68E-04 & 4.41  & 2.18E-05 & 4.08  & 1.41E-06 & 3.96  \\
        $k=4$ & 1.30E-03 & 3.48E-05 & 5.22  & 1.22E-06 & 4.84  & 3.97E-08 & 4.94  \\ \hlineB{2}
              &          &          &       &          &       &          &       \\ \hlineB{2}
              & \multicolumn{7}{c}{$L_\infty$ errors and orders for \textit{the symmetric DDG method}}      \\ \cline{2-8} 
              & $h$      & $h/2$    & Order & $h/4$    & Order & $h/8$    & Order \\ \hline
        $k=2$ & 2.38E-02 & 4.96E-03 & 2.26  & 4.71E-04 & 3.40  & 4.54E-05 & 3.37  \\
        $k=3$ & 8.33E-03 & 4.03E-04 & 4.37  & 2.03E-05 & 4.31  & 1.28E-06 & 3.99  \\
        $k=4$ & 1.37E-03 & 3.61E-05 & 5.25  & 1.27E-06 & 4.82  & 4.17E-08 & 4.94  \\ \hlineB{2}
              &          &          &       &          &       &          &       \\ \hlineB{2}
              & \multicolumn{7}{c}{$L_\infty$ errors and orders for \textit{the nonsymmetric DDG method}}      \\ \cline{2-8} 
              & $h$      & $h/2$    & Order & $h/4$    & Order & $h/8$    & Order \\ \hline
        $k=2$ & 2.13E-02 & 4.40E-03 & 2.28  & 8.01E-04 & 2.46  & 1.69E-04 & 2.25  \\
        $k=3$ & 8.39E-03 & 5.87E-04 & 3.84  & 4.10E-05 & 3.84  & 2.66E-06 & 3.95  \\
        $k=4$ & 1.49E-03 & 3.44E-05 & 5.43  & 1.50E-06 & 4.52  & 6.94E-08 & 4.44  \\ \hlineB{3}
    \end{tabular}
\end{table}

\vspace{0.5cm}
\hypertarget{example:porous-medium-equation-manuf}{}
\noindent\textbf{\ExamplePorousManuf.} 
In this example, we consider \textit{the porous medium equation}
\begin{equation}\label{eqn:porous-medium-equation}
    \frac{\partial U}{\partial t}=\mu\Delta(U^\gamma),
\end{equation}
where $\gamma$ is a model parameter. Note that this equation models a nonlinear diffusion process for $\gamma\ne1$, i.e. coefficients of the diffusion matrix $A(u)$ are functions of $u$:
\begin{equation*}
    A(u)=
    \mu\begin{pmatrix}
    	\gamma u^{\gamma-1} & 0 \\
    	0 & \gamma u^{\gamma-1}
    \end{pmatrix}=\mu\gamma u^{\gamma-1}\mathbb{I},
\end{equation*}
where $\mathbb{I}\in\mathbb{R}^{2\times2}$ is the identity matrix. Notice that the diffusion matrix is diagonal, but for $\gamma>1$, \Cref{eqn:porous-medium-equation} becomes highly nonlinear. In order to assess the performance of the new DDG methods in a highly nonlinear diffusion problem and measure the convergence rate, we solve \Cref{eqn:porous-medium-equation} on $\Omega=[0,1]\times[0,1]$ with $\gamma=3$ and employ the method of manufactured solutions by enforcing the solution
\begin{equation*}
    U(x,y,t)=e^{-8\pi^2\mu t}\sin(2\pi(x+y)).
\end{equation*}
The initial condition for this problem is obtained from this manufactured solution at $t=0$. Also, note that the boundary conditions are periodic. Furthermore, we set $\lambda=0.1$, $\mu=0.01$, $T=1$, and employ quadrature rules that are exact up to polynomials of degree $4k+1$.

The $L_2$ and $L_\infty$ errors are listed in \Cref{table:porous-medium-L2,table:porous-medium-Linf}, respectively. We observe that all DDG methods converge with optimal $(k+1)th$ order accuracy for all reported cases. Although this problem is highly nonlinear, the optimal convergence might be due to the diagonal structure of the nonlinear diffusion matrix $A(u)$.
\begin{table}[h!]
    \centering
    \caption{$L_2$ errors for \protect\hyperlink{example:porous-medium-equation-manuf}{\ExamplePorousManuf} at $T=1$}\label{table:porous-medium-L2}
    \begin{tabular}{cccccc}
        \hlineB{3}
              & \multicolumn{5}{c}{$L_2$ errors and orders for \textit{the DDGIC method}}          \\ \cline{2-6} 
              & $h$      & $h/2$    & Order & $h/4$    & Order \\ \hline
        $k=2$ & 3.19E-03 & 4.02E-04 & 2.99  & 5.00E-05 & 3.01  \\
        $k=3$ & 2.46E-04 & 1.29E-05 & 4.26  & 7.48E-07 & 4.10  \\
        $k=4$ & 1.74E-05 & 5.63E-07 & 4.95  & 1.54E-08 & 5.20  \\ \hlineB{2}
              &          &          &       &          &       \\ \hlineB{2}
              & \multicolumn{5}{c}{$L_2$ errors and orders for \textit{the symmetric DDG method}}      \\ \cline{2-6} 
              & $h$      & $h/2$    & Order & $h/4$    & Order \\ \hline
        $k=2$ & 3.26E-03 & 4.37E-04 & 2.90  & 5.64E-05 & 2.95  \\
        $k=3$ & 2.87E-04 & 1.43E-05 & 4.33  & 8.33E-07 & 4.10  \\
        $k=4$ & 1.85E-05 & 6.12E-07 & 4.91  & 1.69E-08 & 5.18  \\ \hlineB{2}
              &          &          &       &          &       \\ \hlineB{2}
              & \multicolumn{5}{c}{$L_2$ errors and orders for \textit{the nonsymmetric DDG method}}      \\ \cline{2-6} 
              & $h$      & $h/2$    & Order & $h/4$    & Order \\ \hline
        $k=2$ & 2.93E-03 & 4.41E-04 & 2.73  & 5.19E-05 & 3.09  \\
        $k=3$ & 2.09E-04 & 1.31E-05 & 3.99  & 8.54E-07 & 3.94  \\
        $k=4$ & 1.72E-05 & 5.23E-07 & 5.04  & 1.50E-08 & 5.13  \\ \hlineB{3}
    \end{tabular}
\end{table}

\begin{table}[h!]
    \centering
    \caption{$L_\infty$ errors for \protect\hyperlink{example:porous-medium-equation-manuf}{\ExamplePorousManuf} at $T=1$}\label{table:porous-medium-Linf}
    \begin{tabular}{cccccc}
        \hlineB{3}
                 & \multicolumn{5}{c}{$L_\infty$ errors and orders for \textit{the DDGIC method}}          \\ \cline{2-6} 
                 & $h$      & $h/2$    & Order & $h/4$    & Order \\ \hline
        $k=2$    & 2.42E-02 & 3.87E-03 & 2.65  & 5.33E-04 & 2.86  \\
        $k=3$    & 2.29E-03 & 7.78E-05 & 4.88  & 3.45E-06 & 4.49  \\
        $k=4$    & 2.60E-04 & 1.06E-05 & 4.62  & 3.50E-07 & 4.92  \\ \hlineB{2}
                 &          &          &       &          &       \\ \hlineB{2}
                 & \multicolumn{5}{c}{$L_\infty$ errors and orders for \textit{the symmetric DDG method}}      \\ \cline{2-6} 
                 & $h$      & $h/2$    & Order & $h/4$    & Order \\ \hline
        $k=2$    & 2.29E-02 & 3.54E-03 & 2.69  & 4.97E-04 & 2.83  \\
        $k=3$    & 2.05E-03 & 5.61E-05 & 5.19  & 2.78E-06 & 4.33  \\
        $k=4$    & 2.45E-04 & 1.02E-05 & 4.59  & 3.41E-07 & 4.90  \\ \hlineB{2}
                 &          &          &       &          &       \\ \hlineB{2}
                 & \multicolumn{5}{c}{$L_\infty$ errors and orders for \textit{the nonsymmetric DDG method}}      \\ \cline{2-6} 
                 & $h$      & $h/2$    & Order & $h/4$    & Order \\ \hline
        $k=2$ & 2.69E-02    & 4.46E-03 & 2.59  & 6.06E-04 & 2.88  \\
        $k=3$ & 2.45E-03    & 6.97E-05 & 5.14  & 4.32E-06 & 4.01  \\
        $k=4$ & 2.79E-04    & 1.15E-05 & 4.61  & 3.70E-07 & 4.95  \\ \hlineB{3}
    \end{tabular}
\end{table}
\vspace{0.5cm}
\hypertarget{example:porous-medium-equation-smooth}{}
\noindent\textbf{\ExamplePorousSmooth.} In this example, we consider the same equation as in \hyperlink{example:porous-medium-equation-manuf}{\ExamplePorousManuf}, but with the model coefficient $\gamma=2$ on $\Omega=[-10,10]\times[-10,10]$. Here, we investigate the performance of the new DDG methods for two initially disconnected, merging bumps which are defined by the following: 
\begin{equation*}
    U_0(x,y)=
    \begin{cases}
    e^{\frac{-1}{6-(x-2)^2-(y+2)^2}}, & (x-2)^2-(y+2)^2<6 \\
    e^{\frac{-1}{6-(x+2)^2-(y-2)^2}}, & (x+2)^2-(y-2)^2<6 \\
    0, & \text{otherwise}.
    \end{cases}
\end{equation*}
Note that $U(x,y,t)=0$ on $\partial\Omega$ for $t\ge 0$. In this case, we solve the problem on $h/16$ mesh along with $\lambda=0.1$, $\mu=1$ and $T=4$. In \Cref{fig:porous1}, we present a third order ($k=2$) symmetric DDG solution. The solutions corresponding to other DDG versions are similar, hence they are not included. In this case, we solve the problem on $h/16$ mesh along with $\lambda=0.1$, $\mu=1$ and $T=4$. 
We observe that bumps are diffused quickly and merged with finite time. The results are in good agreement with those in literature \cite{vidden2013sddg,yan2013,liu2011high}
\begin{figure}[h!]
    \begin{center}
        \begin{subfigure}{0.48\textwidth}
            \centering
            \includegraphics[scale=0.3,trim={0.5cm 0.5cm 0.5cm 2.25cm},clip]{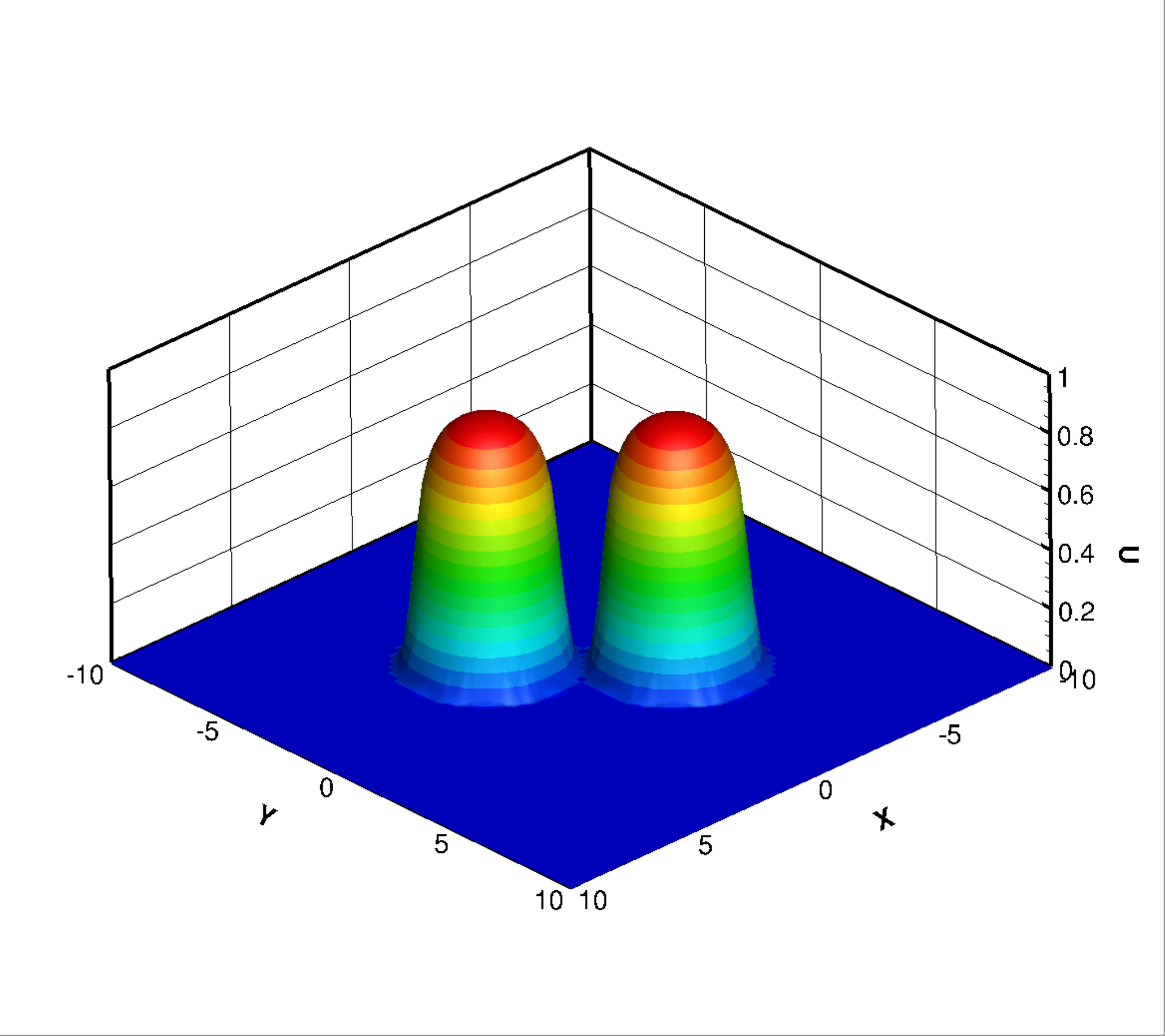}
            \caption{$t=0$}
        \end{subfigure}
        \begin{subfigure}{0.48\textwidth}
            \centering
            \includegraphics[scale=0.3,trim={0.5cm 0.5cm 0.5cm 2.25cm},clip]{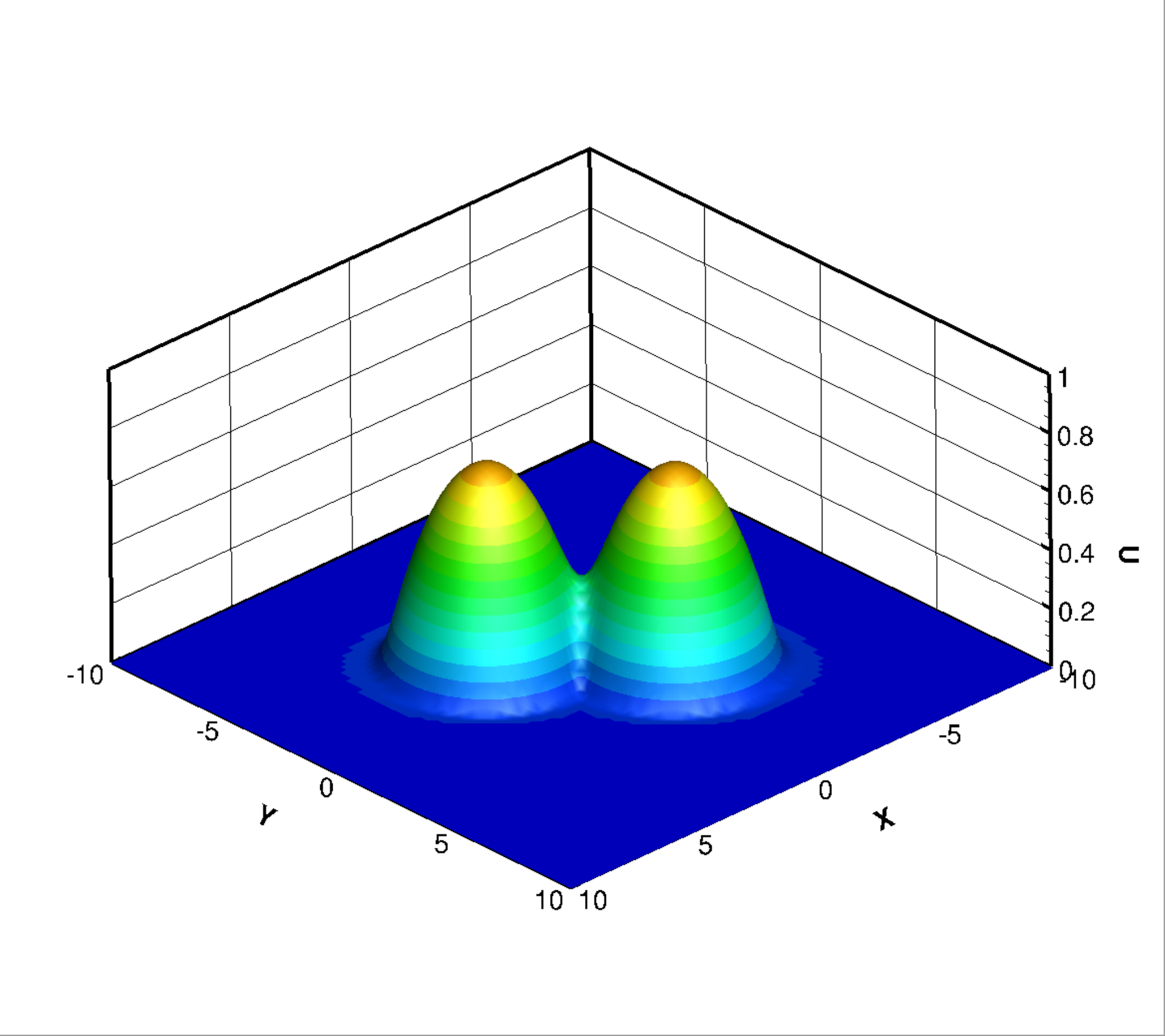}
            \caption{$t=0.5$}
        \end{subfigure}
        \begin{subfigure}{0.48\textwidth}
            \centering
            \includegraphics[scale=0.3,trim={0.5cm 2.65cm 0.5cm 1.5cm},clip]{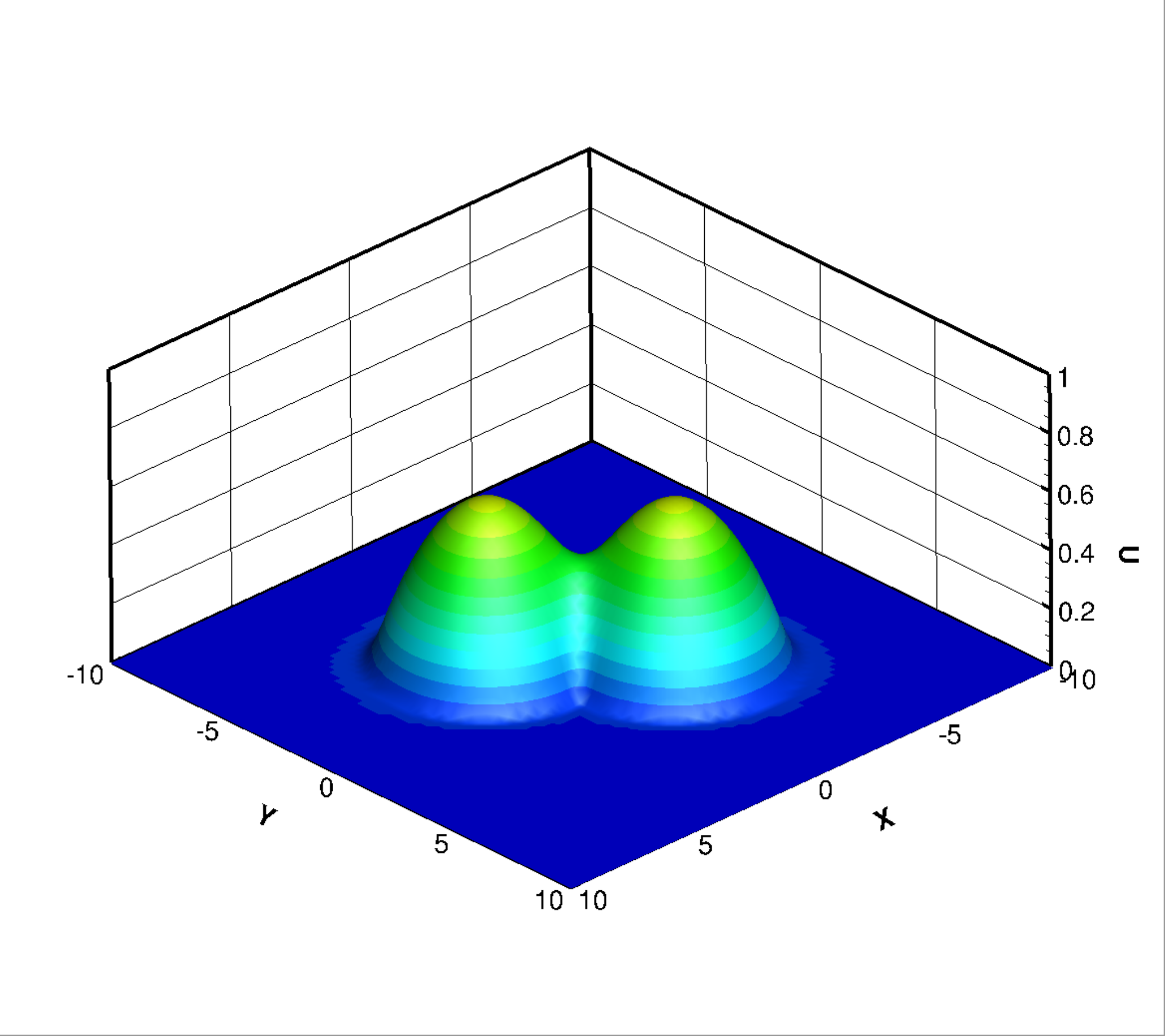}
            \caption{$t=1$}
        \end{subfigure}
        \begin{subfigure}{0.48\textwidth}
            \centering
            \includegraphics[scale=0.3,trim={0.5cm 2.65cm 0.5cm 1.5cm},clip]{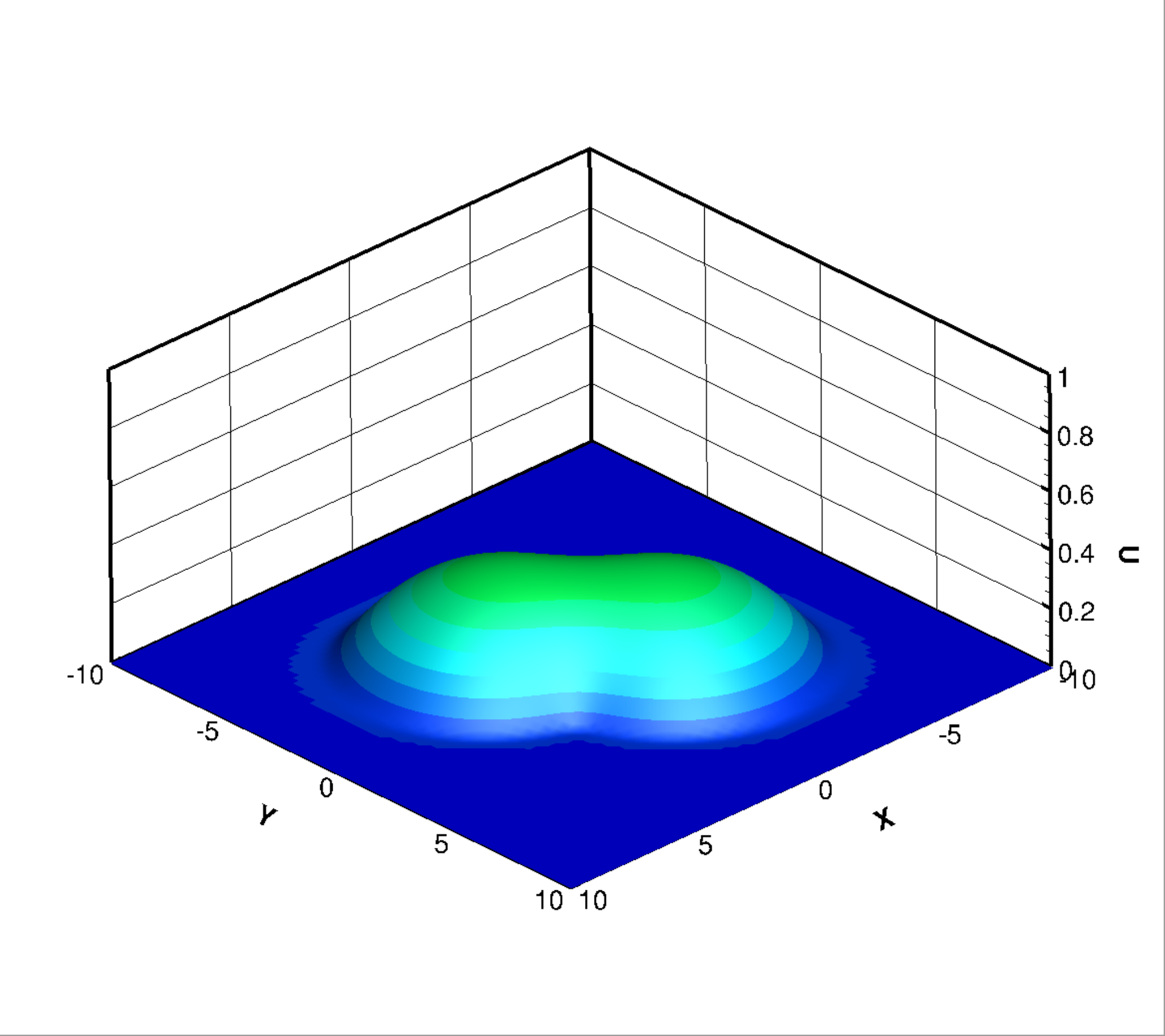}
            \caption{$t=4$}
        \end{subfigure}
    \end{center}
    \caption{The contour plot corresponding to the third ($k=2$) order symmetric DDG solution for \protect\hyperlink{example:porous-medium-equation-smooth}{\ExamplePorousSmooth}. 17 equally spaced contours  are used between $0$ and $0.8$.}
    \label{fig:porous1}
\end{figure}

\vspace{0.5cm}
\hypertarget{example:porous-medium-equation-discontinous}{}
\noindent\textbf{\ExamplePorousDiscontin} In this example, we continue to study the same equation as in \hyperlink{example:porous-medium-equation-manuf}{\ExamplePorousManuf}, but with the model coefficient $\gamma=2$ on $\Omega=[-1,1]\times[-1,1]$. $U(x,y,t)=0$ on $\partial\Omega$ for $t\ge 0$ and the initial condition is defined as
\begin{equation*}
    U_0(x,y)=
    \begin{cases}
    1, & (x,y)\in [-\frac{1}{2},\frac{1}{2}]\times[-\frac{1}{2},\frac{1}{2}] \\
    0, & \text{otherwise}.
    \end{cases}
\end{equation*}
We set $\lambda=0.1$, $\mu=1$ and $T=0.005$ and solve the problem on $h/16$ mesh. 

Since the initial condition is discontinuous, the numerical solutions obtained by the new DDG methods blow up unless a maximum-principle-satisfying (MPS) limiter is employed. Therefore, the linear scaling limiter in \cite{zhang2010_mps} is implemented to keep the numerical solution in $0\le u(x,y,t) \le 1$ for $t\ge0$. Note that the CFL condition \Cref{eqn:cfl-condition} is still in use, and the DDG parameters $\beta_0$ and $\beta_1$ are kept the same. 

In \Cref{fig:porous-medium-discontinuous}, the numerical solution obtained by the DDGIC method for a third order numerical solution ($k=2$) is shown. We observe that the numerical solution diffuses out smoothly in a stable manner and is in good agreement with Example 5.7 of \cite{DuMPS2019}.  
\begin{figure}[h!]
    \begin{center}
    	\begin{subfigure}[b]{0.48\textwidth}
    		\centering
    		\includegraphics[scale=0.3,trim={1cm 0.5cm 1cm 1.75cm},clip]{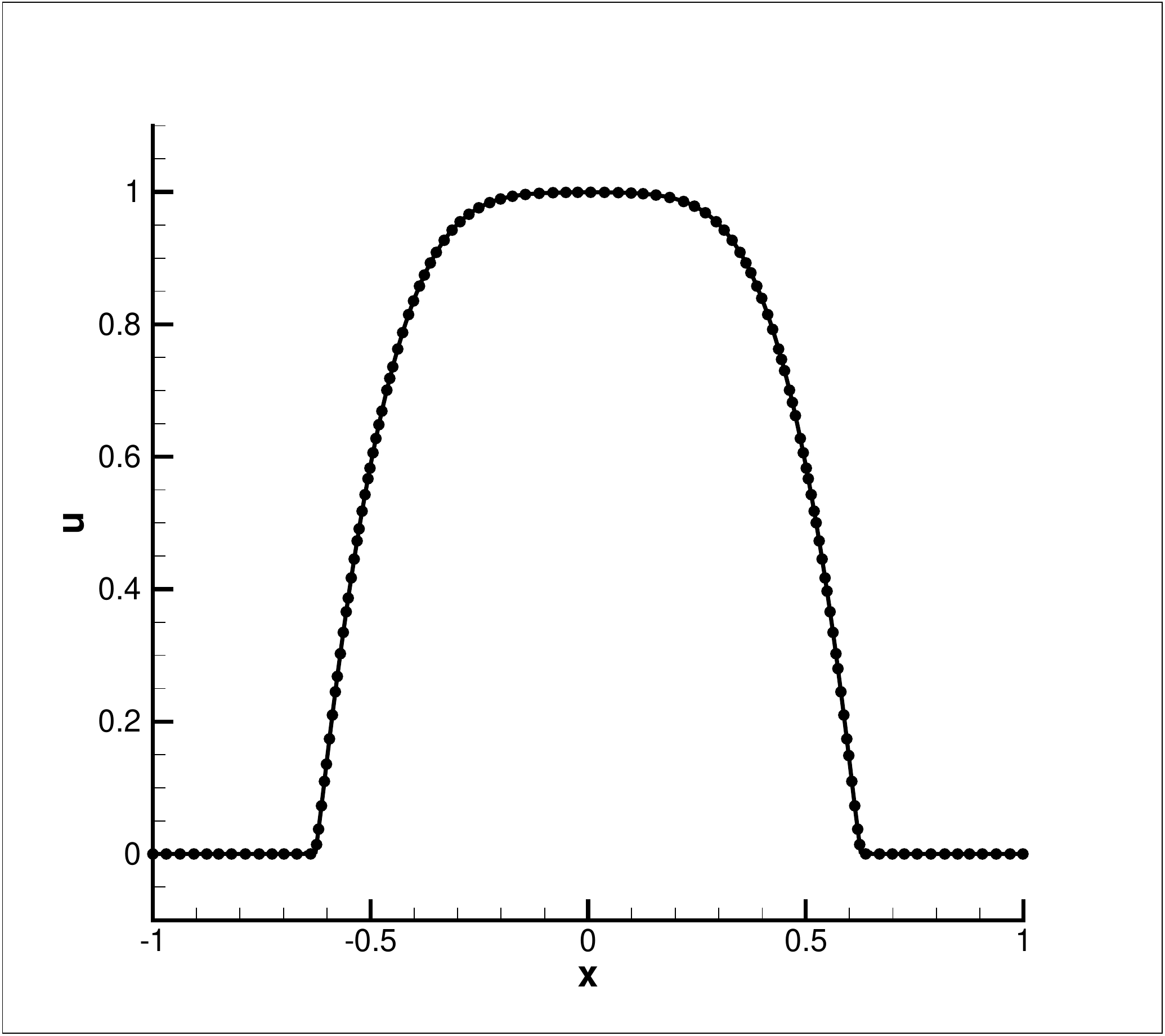}
    		\caption{}
    	\end{subfigure}
    	\begin{subfigure}[b]{0.48\textwidth}
    		\centering
    		\includegraphics[scale=0.3,trim={1cm 0.5cm 1cm 1.75cm},clip]{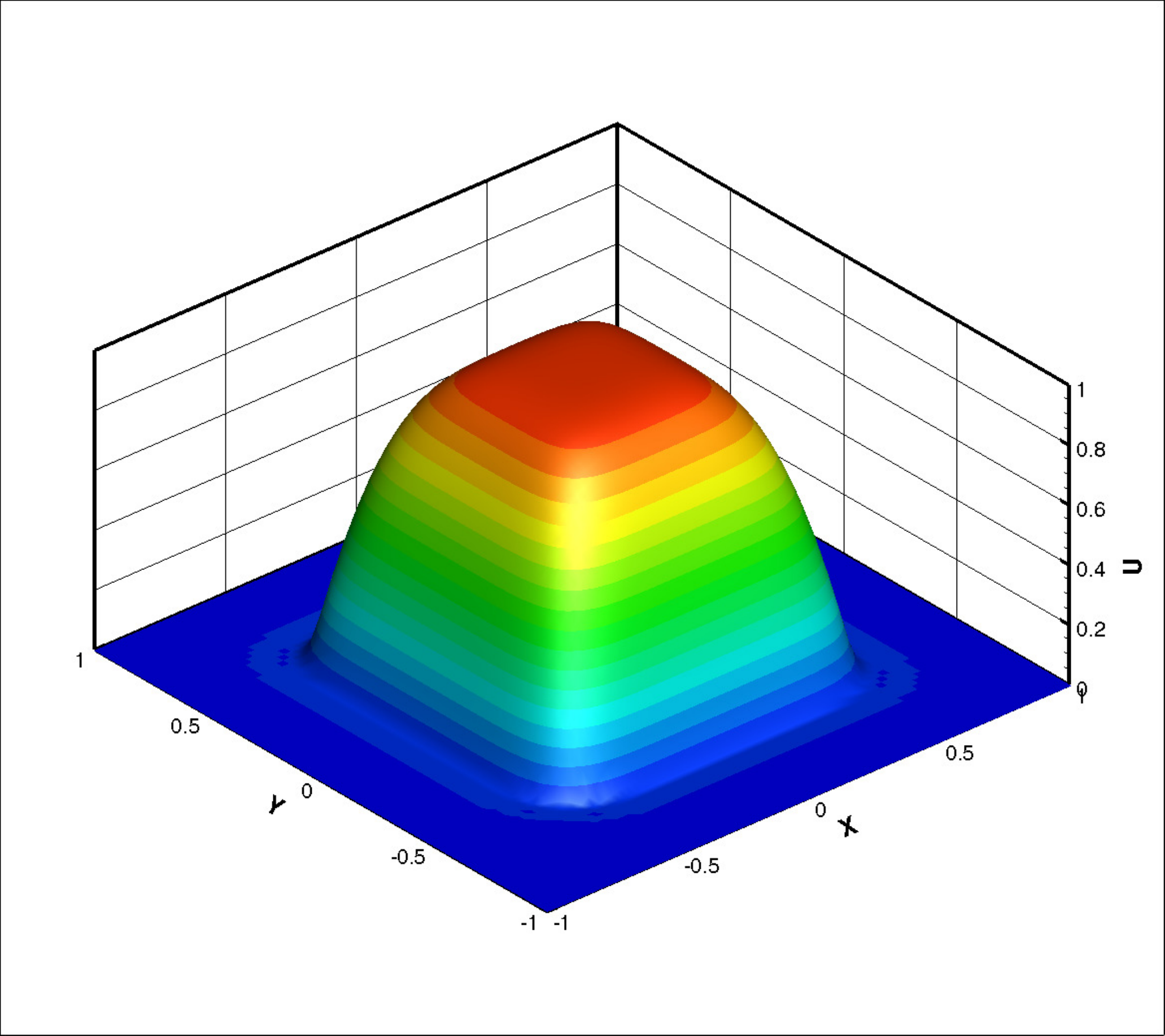}
    		\caption{}
    	\end{subfigure}
    \end{center}
	\caption{The third ($k=2$) order numerical solution obtained by DDGIC with MPS limiter for \protect\hyperlink{example:porous-medium-equation-smooth}{\ExamplePorousSmooth} at $T=0.005$. (a) The solution along $y=0$, (b) surface plot. 19 equally spaced contour levels  are used between 0 and 1. }\label{fig:porous-medium-discontinuous}
\end{figure}

\vspace{0.5cm}
\hypertarget{example:blwp}
\noindent\textbf{\ExampleBlwp} In this example, we consider 
\begin{equation}\label{eqn:blowup}
    \frac{\partial U}{\partial t}=\mu\left(2U_{xx}+3(U^{1.5})_{yy}+U^2\right),
\end{equation}
with the initial condition $U_0(x,y)=200\sin{(\pi x)}\sin{(\pi y)}$ on $\Omega=[0,1]\times[0,1]$. We apply the homogeneous Dirichlet boundary condition. Due to the last term on the right-hand side of \Cref{eqn:blowup}, the solution eventually blows up. As reported in Example 3.7 of  \cite{GUO2015181}, we observe that the numerical solution blows up after one time-step when a positivity-preserving limiter is not used. Therefore, the linear scaling limiter in \cite{zhang2010_mps} is employed to maintain the stability of the numerical approximation. However, the limiter's upper bound condition on the numerical solution is removed and the lower bound is set to 0 to preserve the positivity of the numerical solution. Furthermore, according to \cite{GUO2015181}, the CFL condition in \Cref{eqn:cfl-condition} is modified as
\begin{equation*}
    \Delta t\frac{\mu}{\min_K{h^2_K}}<\min{\left(\omega\lambda,\frac{1}{\max_K{u}}\right)}.
\end{equation*}
As in the previous example, the DDG parameters $\beta_0,\beta_1$ are kept the same. The $h/16$ mesh is used to solve this problem, and we set $\lambda=0.01$, $\mu=1$, and the quadrature rules are exact up to polynomials of degree $(2k+1)th$. As for the stopping criterion, we follow \cite{GUO2015181} and take $\Delta t<10^{-13}$. However, it is worth emphasizing that the focus of this paper is not to design a positivity-preserving limiter. Instead, we simply explore the performance of the new DDG methods in computationally challenging cases. Therefore, the modified CFL condition along with the linear scaling limiter does not guarantee the positivity of the numerical solution for the time level $t^{n+1}$. In fact, almost all of the MPS or positivity-preserving limiters are designed for forward Euler method, but instead, they are used in a Runge-Kutta scheme. Even though they might propose a CFL condition that might guarantee the positivity (or maximum-principle) of the numerical solution for a forward Euler method, the physical limits of the problem might be violated by the numerical solution at any stage of the Runge-Kutta scheme. Therefore, we follow \cite{Zhang-2017-NS} and restart the Runge-Kutta time step with $\Delta t/2$ when it is no longer possible to maintain the positivity of the numerical solution in one or more cells.

For a third ($k=2$) order solution, we observe that the restart algorithm is only activated at the last time-step at $t=1.8155\times10^{-2}$ for all DDG versions. However, it turns out that the time step becomes smaller than the machine precision $\epsilon$, i.e. $\Delta t<\epsilon$, during the restart procedure. This means that it is no longer possible to continue time-stepping without violating the positivity in one or more cells. Therefore, we denote $t=1.81552\times10^{-2}$ as the blow-up time. Note that this value is slightly smaller than the blow-up time $t=1.82378\times10^{-2}$ reported in \cite{GUO2015181}. Thus, we obtain slightly lower values for $\max_K{u_K}$. In \Cref{fig:blowup}, we show the numerical solution for a third order ($k=2$) numerical solution obtained by the DDGIC method when the solution blows up at $t=1.81552\times10^{-2}$. Since the results obtained by other DDG versions are similar, they are not shown. We observe that the numerical solution does not have oscillations and qualitatively similar to that in \cite{GUO2015181}. 
\begin{figure}[h!]
	\begin{center}
    	\begin{subfigure}[b]{0.48\textwidth}
    		\centering
    		\includegraphics[scale=0.35,trim={2cm 0.10cm 2cm 0.1cm},clip]{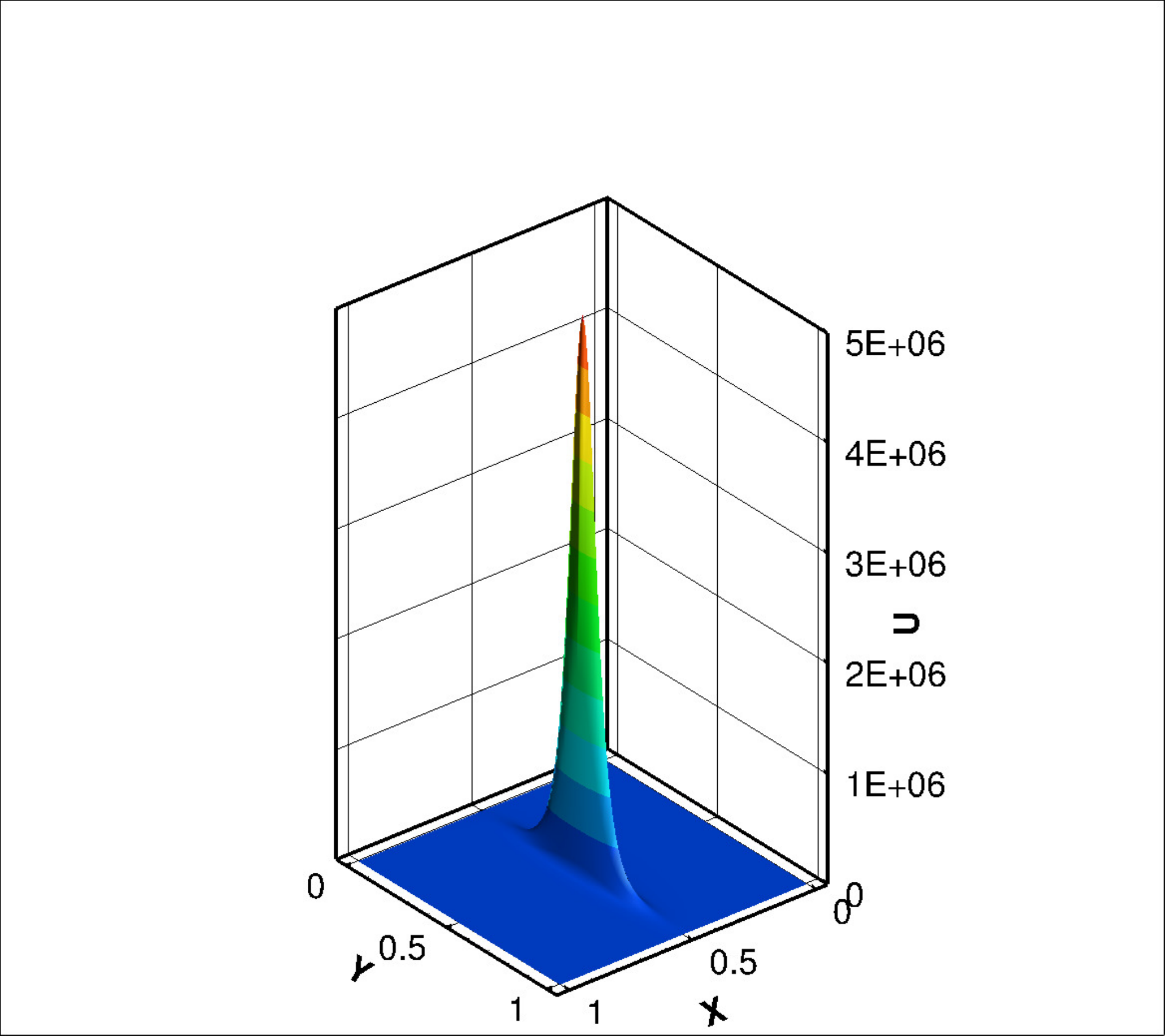}
    		\caption{u}
    	\end{subfigure}
    	\begin{subfigure}[b]{0.48\textwidth}
    		\centering
    		\includegraphics[scale=0.35,trim={2cm 0.10cm 2cm 0.1cm},clip]{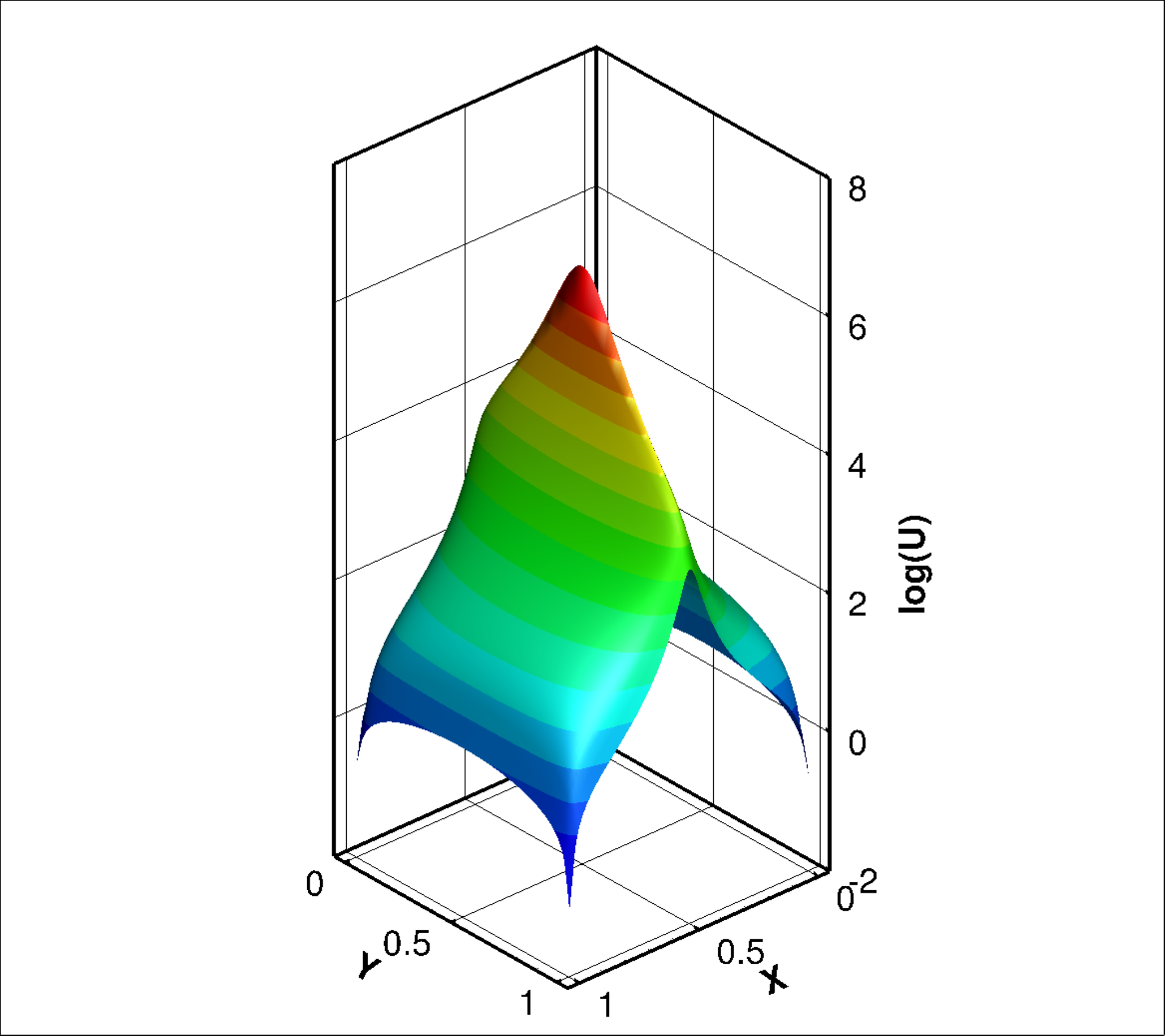}
    		\caption{log(u)}
    	\end{subfigure}
    \end{center}
    \caption{The third ($k=2$) order numerical solution obtained by DDGIC with MPS limiter for \protect\hyperlink{example:blowup}{\ExampleBlwp} at $T=1.81552\times10^{-2}$}\label{fig:blowup}
\end{figure}

\section{Concluding Remarks}\label{sec:conclusion}

In this paper, we have unified the new framework for direct discontinuous Galerkin (DDG) methods by extending the new DDGIC method \cite{danis2021new} to the symmetric and the nonsymmetric DDG versions. Unlike their original counterparts, the new DDG methods do not require evaluating an antiderivative of the nonlinear diffusion matrix. By constructing a new direction vector at each element interface, a linear numerical flux is used regardless of the problem type. Furthermore, the nonlinear linear stability theory of the new methods has been developed and several numerical experiments have been conducted to perform the error analysis. In all numerical examples, the DDGIC and symmetric DDG methods demonstrated optimal $(k+1)th$ order convergence and their performances were assessed to be equivalent. On the other hand, the performance of the nonsymmetric DDG method varied in all numerical examples. The nonsymmetric DDG method achieved optimal $(k+1)th$ order convergence for odd degree polynomials in all examples while an order loss was observed with the even degree polynomials except for the cases with diagonal diffusion matrices. In short, a high-order accuracy is achieved by all new DDG methods. However, the DDGIC and symmetric DDG methods were found to be superior to the nonsymmetric version. 
\appendix
\section{Important Inequalities}\label{sec:appendix} 

In this section, we discuss important inequalities used in the proofs of the stability analysis for symmetric DDG and DDGIC methods. 

\begin{lemma}[Young's inequality]\label{lemma:youngs_ineq}
Suppose that $a,b \geq 0$, $1<p,q,\infty$, and that $\frac{1}{p} + \frac{1}{q}=1$. Then, we have that
\begin{equation*}
    ab\leq \frac{a^p}{p} + \frac{b^q}{q}.
\end{equation*}
\end{lemma}
A corollary to \Cref{lemma:youngs_ineq} can be obtained by considering $ab=(a\epsilon^{1/p})(b/\epsilon^{1/p})$ for $\epsilon>0$.
\begin{corollary}\label{cor:youngs_ineq}
Suppose that $a,b \geq 0$, $1<p,q,\infty$, and that $\frac{1}{p} + \frac{1}{q}=1$. Furthermore, if $\epsilon>0$, then 
\begin{equation*}
    ab\leq \frac{\epsilon a^p}{p} + \frac{b^q}{q\epsilon^{q/p}}.
\end{equation*}
\end{corollary}

We will also recall the following lemma due to Ern and Guermord \cite{ern2021finite}:
\begin{lemma}\label{lemma:bound-on-sobolev-seminorm}
Let $v\in\VhK$ and $l\in\mathbb{N}$. There exists a constant $C>0$ for any non-negative integer $m \leq l$ such that
\begin{equation*}\label{eqn:bound-on-sobolev-seminorm}
    \left| v \right|_{W^{l,p}(K)} \leq Ch_K^{m-l+d\left(\frac{1}{p}-\frac{1}{r}\right)}\left| v \right|_{W^{m,r}(K)}, \quad\forall p,r\in[1,\infty].
\end{equation*}
\end{lemma}
\begin{proof}
See the proof of Lemma 12.1 in \cite{ern2021finite}.
\end{proof}

Next, we will derive a series of essential inequalities used in the proof of the stability results for symmetric DDG and DDGIC methods.
\begin{lemma} \label{lemma:xi_ineq}
Assume that $A(u)\in\mathbb{R}^{2\times 2}$ is positive definite and there exist $\gamma,\gamma^*\in\mathbb{R}$ such that the eigenvalues $(\gamma_1,\gamma_2)$ of $A(u)$ lie between $[\gamma,\gamma^*]$ for $\forall u\in\mathbb{R}$.
If $\bxi(u)$ is given as in \Cref{eqn:discrete-new-direction-vector}, then $\forall\mathbf{x}\in\mathbb{R}^2$ there holds
\begin{equation*}
    \left|\bxi(u)\cdot\mathbf{x}\right|\leq\gamma^*\left\|\mathbf{x}\right\|,
\end{equation*}
where we denote by $\left\|\cdot\right\|$ the Euclidean norm in $\mathbb{R}^2$.
\end{lemma}
\begin{proof}
Using the Scharwz inequality, we have
\begin{equation*}
    \left|\bxi(u)\cdot\mathbf{x}\right|\leq \left\|\bxi(u) \right\|\left\|\mathbf{x} \right\|.
\end{equation*}
Let $\mathbf{e}_1,\mathbf{e}_2$ be the orthonormal eigenvectors corresponding to the eigenvalues $\gamma \text{ and } \gamma^*$, respectively. Since $A(u)$ is positive-definite, its eigenvectors form a basis of $\mathbb{R}^2$. Then, the unit normal vector can be written as $\mathbf{n}=n_1 \mathbf{e}_1 + n_2 \mathbf{e}_2$, and then, it follows that
\begin{equation*}
    \begin{aligned}
        \left\|\bxi(u) \right\|^2 &= \left\|A(\llkh u \rrkh)^T\mathbf{n}\right\|^2 = \left\|\gamma_1 n_1 \mathbf{e}_1 + \gamma_2 n_2 \mathbf{e}_2\right\|^2 = \gamma_1^2 n_1^2+\gamma_2^2 n_2^2 \\
        & \leq (\gamma^*)^2\left(n_1^2+n_2^2\right) = (\gamma^*)^2
    \end{aligned}
\end{equation*}
and the conclusion holds.
\end{proof}

\begin{lemma} \label{lemma:surf_jump}
Suppose that $A(u)\in\mathbb{R}^{2\times 2}$ is positive definite and there exist $\gamma,\gamma^*\in\mathbb{R}$ such that the eigenvalues $(\gamma_1,\gamma_2)$ of $A(u)$ lie between $[\gamma,\gamma^*]$ for $\forall u\in\mathbb{R}$. If $\beta_0 \geq 0$, then we have that
\begin{equation*}
    \sum_{e\in\mathcal{E}_h}\int_{e} \frac{\beta_0}{h_e} \llbracket u \rrbracket^2 \mathbf{n}\cdot\bxi(\llkh u \rrkh)~ds \geq \sum_{e\in\mathcal{E}_h}\int_{e}\frac{\gamma\beta_0}{h_e} \llbracket u \rrbracket^2 ~ ds.
\end{equation*}
\end{lemma}

\begin{proof}
Recall the definition of the new direction vector $\bxi\llkh u \rrkh=A(\llkh u \rrkh)^T\mathbf{n}$. Then,
\begin{equation*}
    \frac{\beta_0}{h_e} \llbracket u \rrbracket^2 \mathbf{n}\cdot\bxi(\llkh u \rrkh)=\frac{\beta_0}{h_e} \llbracket u \rrbracket^2 \mathbf{n}\cdot A(\llkh u \rrkh)^T\mathbf{n}.
\end{equation*}
Since $\mathbf{n}\cdot A(\llkh u \rrkh)^T\mathbf{n} \geq \gamma_1 \geq \gamma $, the conclusion follows.
\end{proof}

\begin{lemma} \label{lemma:surf_avg}
Suppose that $A(u)\in\mathbb{R}^{2\times 2}$ is positive definite and there exist $\gamma,\gamma^*\in\mathbb{R}$ such that the eigenvalues $(\gamma_1,\gamma_2)$ of $A(u)$ lie between $[\gamma,\gamma^*]$ for $\forall u\in\VhK$. If $\beta_0 \geq 0$, then there exists a constant $C>0$ such that
\begin{equation*}
    \sum_{e\in\mathcal{E}_h}\int_{e} \llbracket u \rrbracket\llkh \nabla u \rrkh\cdot\bxi(\llkh u \rrkh) ~ds\geq -\sum_{e\in\mathcal{E}_h}\frac{\gamma\beta_0}{2h}\left\| \llbracket u \rrbracket \right\|^2_{L^2(e)} -\sum_{K\in\mathcal{T}_h} C\frac{(\gamma^*k)^2}{4\gamma\beta_0}\left\| \nabla u \right\|^2_{L^2(K)}.
\end{equation*}
Furthermore, under the same assumptions, there also holds
\begin{equation*}
    \sum_{e\in\mathcal{E}_h}\int_{e} \llbracket u \rrbracket\llkh \nabla u \rrkh\cdot\bxi(\llkh u \rrkh) ~ds\geq -\sum_{e\in\mathcal{E}_h}\frac{\gamma\beta_0}{4h}\left\| \llbracket u \rrbracket \right\|^2_{L^2(e)} -\sum_{K\in\mathcal{T}_h} C\frac{(\gamma^*k)^2}{2\gamma\beta_0}\left\| \nabla u \right\|^2_{L^2(K)}.
\end{equation*}
\end{lemma}

\begin{proof}
Note that $\llkh \nabla u \rrkh=\frac{1}{2}(\nabla u)^+ + \frac{1}{2}(\nabla u)^-$. Then, we have
\begin{equation*}
    \begin{aligned}
        \sum_{e\in\mathcal{E}_h}\int_{e} \llbracket u \rrbracket\llkh \nabla u \rrkh&\cdot\bxi(\llkh u \rrkh) ~ds \\ &= \frac{1}{2}\sum_{e\in\mathcal{E}_h}\int_{e} \llbracket u \rrbracket\left((\nabla u)^++(\nabla u)^-\right)\cdot\bxi(\llkh u \rrkh) ~ds \\
        &\geq -\frac{1}{2}\sum_{e\in\mathcal{E}_h}\int_{e} \left|\llbracket u \rrbracket\right| \left(\left|(\nabla u)^+\cdot\bxi(\llkh u \rrkh)\right|+ \left|(\nabla u)^-\cdot\bxi(\llkh u \rrkh)\right|\right) ~ds \\
        &\geq  -\frac{\gamma^*}{2}\sum_{e\in\mathcal{E}_h}\int_{e} \left|\llbracket u \rrbracket\right| \left\|(\nabla u)^+\right\| ~ds -\frac{\gamma^*}{2}\sum_{e\in\mathcal{E}_h}\int_{e} \left|\llbracket u \rrbracket\right| \left\|(\nabla u)^-\right\|~ds ,
    \end{aligned}
\end{equation*}
where we have invoked \Cref{lemma:xi_ineq} in the last step. Moreover, by \Cref{cor:youngs_ineq} with $\epsilon=\frac{\gamma\beta_0}{\gamma^*h}$, we obtain
\begin{equation}\label{eqn:surf_avg_youngs_ineq}
    \gamma^*\int_{e} \left|\llbracket u \rrbracket\right|\, \left\|(\nabla u)^\pm\right\| ~ds \leq \frac{\gamma\beta_0}{2h}\left\| \llbracket u \rrbracket \right\|^2_{L^2(e)} + \frac{(\gamma^*)^2h}{2\gamma\beta_0}\left\| (\nabla u)^\pm \right\|^2_{L^2(e)}.
\end{equation}
So that we have
\begin{equation}\label{eqn:surf_avg_intermediate_ineq}
    \begin{aligned}
        \sum_{e\in\mathcal{E}_h}\int_{e} \llbracket u \rrbracket\llkh \nabla u \rrkh&\cdot\bxi(\llkh u \rrkh) ~ds \\ &\geq -\sum_{e\in\mathcal{E}_h}\frac{\gamma\beta_0}{2h}\left\| \llbracket u \rrbracket \right\|^2_{L^2(e)} \\ &-\sum_{e\in\mathcal{E}_h} \frac{(\gamma^*)^2h}{4\gamma\beta_0}\left\| (\nabla u)^+ \right\|^2_{L^2(e)}-\sum_{e\in\mathcal{E}_h} \frac{(\gamma^*)^2h}{4\gamma\beta_0}\left\| (\nabla u)^- \right\|^2_{L^2(e)}.
    \end{aligned}
\end{equation}
Note that the last two terms above are simply summations over individual edges in the triangulation $\mathcal{T}_h$, and each summation is responsible for accumulating $\left\| (\nabla u)^\pm \right\|^2_{L^2(e)}$ only from one side of the edge. In the global sense, these summations accumulate $\left\| (\nabla u)^{interior} \right\|^2_{L^2(\partial K)}$ for each cell $K$ in the domain. Thus, they can be converted into a single summation over cells. That is, 
\begin{equation}\label{eqn:surf_avg_equivalent_sums}
    \sum_{e\in\mathcal{E}_h} \frac{(\gamma^*)^2h}{4\gamma\beta_0}\left\| (\nabla u)^+ \right\|^2_{L^2(e)}+\sum_{e\in\mathcal{E}_h} \frac{(\gamma^*)^2h}{4\gamma\beta_0}\left\| (\nabla u)^- \right\|^2_{L^2(e)} = \sum_{K\in\mathcal{T}_h} \frac{(\gamma^*)^2h}{4\gamma\beta_0}\left\| \nabla u \right\|^2_{L^2(\partial K)}.
\end{equation}
This expression is useful since we can invoke the trace inequality
\begin{equation} \label{eqn:surf_avg_trace_ineq}
    \sum_{K\in\mathcal{T}_h} \frac{(\gamma^*)^2h}{4\gamma\beta_0}\left\| \nabla u \right\|^2_{L^2(\partial K)} \leq \sum_{K\in\mathcal{T}_h} C\frac{(\gamma^*k)^2}{4\gamma\beta_0}\left\| \nabla u \right\|^2_{L^2(K)},
\end{equation}
for some constant $C>0$. Thus, substituting \Cref{eqn:surf_avg_equivalent_sums,eqn:surf_avg_trace_ineq} into \Cref{eqn:surf_avg_intermediate_ineq} completes the first part of the proof. The second part follows after following the same steps but using $\epsilon=\frac{\gamma\beta_0}{2\gamma^*h}$ in \Cref{eqn:surf_avg_youngs_ineq} instead.
\end{proof}

\begin{lemma}\label{lemma:surf_jump_second_deriv}
Suppose that $A(u)\in\mathbb{R}^{2\times 2}$ is positive definite and there exist $\gamma,\gamma^*\in\mathbb{R}$ such that the eigenvalues $(\gamma_1,\gamma_2)$ of $A(u)$ lie between $[\gamma,\gamma^*]$ for $\forall u\in\VhK$. If $\beta_0 \geq 0$, then there exists a constant $C>0$ such that
\begin{equation*}
    \begin{aligned}
        \sum_{e\in\mathcal{E}_h}\int_{e} \beta_1 h\llbracket u \rrbracket\llbracket\nabla (\nabla u\cdot\mathbf{n}) \rrbracket&\cdot\bxi(\llkh u \rrkh) ~ds \\ &\geq - \sum_{e\in\mathcal{E}_h}\frac{\gamma\beta_0}{2h}\left\| \llbracket u \rrbracket \right\|^2_{L^2(e)} 
        -\sum_{K\in\mathcal{T}_h} C\frac{(\gamma^*\beta_1k)^2}{\gamma\beta_0} \left\|\nabla u \right\|^2_{L^2(K)}. 
    \end{aligned}
\end{equation*}
\end{lemma}

\begin{proof}
By convention, the outward unit normal vector $\mathbf{n}$ is understood as $\mathbf{n}=\mathbf{n}^+$. Also, it can be understood in terms of the inward unit normal vector as $\mathbf{n}=-\mathbf{n}^-$. Therefore, the jump term for the second derivatives can be rewritten as
\begin{equation*}
    \llbracket\nabla (\nabla u\cdot\mathbf{n})\rrbracket= \left(\nabla (\nabla u\cdot\mathbf{n})\right)^+ + \left(\nabla (\nabla u\cdot\mathbf{n})\right)^-.
\end{equation*}
With this understanding, we have that 
\begin{equation*}
    \begin{aligned}
        \int_{e} \beta_1 h\llbracket u \rrbracket\llbracket\nabla (\nabla u\cdot\mathbf{n}) \rrbracket&\cdot\bxi(\llkh u \rrkh) ~ds \\ &= 
        \int_{e} \beta_1 h\llbracket u \rrbracket\left((\nabla (\nabla u\cdot\mathbf{n}) )^+ + (\nabla (\nabla u\cdot\mathbf{n}) )^-\right)\cdot\bxi(\llkh u \rrkh) ~ds
        \\
        &\geq 
        -\int_{e} \beta_1 h\left|\llbracket u \rrbracket\right|\left|\left((\nabla (\nabla u\cdot\mathbf{n}) )^+ + (\nabla (\nabla u\cdot\mathbf{n}) )^-\right)\cdot\bxi(\llkh u \rrkh)\right| ~ds
        \\
        &\geq 
        -\gamma^*\int_{e} \beta_1 h\left|\llbracket u \rrbracket\right|\left(\left\|(\nabla (\nabla u\cdot\mathbf{n}) )^+\right\| + \left\|(\nabla (\nabla u\cdot\mathbf{n}) )^-\right\|\right) ~ds.
    \end{aligned}
\end{equation*}
Note that we have invoked \Cref{lemma:xi_ineq} and triangle inequality in the last step. Furthermore, by \Cref{cor:youngs_ineq} with $\epsilon=\frac{\gamma\beta_0}{2\gamma^*\beta_1h^2}$, we obtain
\begin{equation*}
    \gamma^*\int_{e} \beta_1 h \left|\llbracket u \rrbracket\right| \left\|(\nabla (\nabla u\cdot\mathbf{n}) )^\pm \right\| ds \leq \frac{\gamma\beta_0}{4h}\left\| \llbracket u \rrbracket \right\|^2_{L^2(e)} + \frac{(\gamma^*\beta_1)^2h^3}{\gamma\beta_0} \left\|(\nabla (\nabla u\cdot\mathbf{n}) )^\pm \right\|^2_{L^2(e)}.
\end{equation*}
Thus, 
\begin{equation}\label{eqn:surf_jump_second_deriv_intermediate_ineq}
    \begin{aligned}
        \sum_{e\in\mathcal{E}_h}\int_{e} \beta_1 h\llbracket u \rrbracket\llbracket\nabla (\nabla u\cdot\mathbf{n}) \rrbracket\cdot\bxi(\llkh u \rrkh) ~ds &\geq - \sum_{e\in\mathcal{E}_h}\frac{\gamma\beta_0}{2h}\left\| \llbracket u \rrbracket \right\|^2_{L^2(e)} \\ 
        &-\sum_{e\in\mathcal{E}_h}\frac{(\gamma^*\beta_1)^2h^3}{\gamma\beta_0} \left\|(\nabla (\nabla u\cdot\mathbf{n}) )^+ \right\|^2_{L^2(e)}  \\
        &-\sum_{e\in\mathcal{E}_h}\frac{(\gamma^*\beta_1)^2h^3}{\gamma\beta_0} \left\|(\nabla (\nabla u\cdot\mathbf{n}) )^- \right\|^2_{L^2(e)}.
    \end{aligned}
\end{equation}
As in the proof of \Cref{lemma:surf_avg}, we convert the summations over edges to a summation over cells:
\begin{equation}\label{eqn:surf_jump_second_deriv_equiv}
    \begin{aligned}
        \sum_{e\in\mathcal{E}_h}\frac{(\gamma^*\beta_1)^2h^3}{\gamma\beta_0} \left\|(\nabla (\nabla u\cdot\mathbf{n}) )^+ \right\|^2_{L^2(e)} &+ \sum_{e\in\mathcal{E}_h}\frac{(\gamma^*\beta_1)^2h^3}{\gamma\beta_0} \left\|(\nabla (\nabla u\cdot\mathbf{n}) )^- \right\|^2_{L^2(e)} \\
        &= \sum_{K\in\mathcal{T}_h} \frac{(\gamma^*\beta_1)^2h^3}{\gamma\beta_0} \left\|\nabla (\nabla u\cdot\mathbf{n})\right\|^2_{L^2(\partial K)}.
    \end{aligned}
\end{equation}
At this point, it might be tempting to invoke the trace theorem for the norm on the right-hand side of the above equation. However, a more useful inequality can be obtained by considering the Euclidean norm $\left\|\nabla (\nabla u\cdot\mathbf{n}) \right\|^2$. We first note that
\begin{equation*}
    \begin{aligned}
        \left\|\nabla (\nabla u\cdot\mathbf{n}) \right\|^2 &
        =(u_{xx}n_1 + u_{yx}n_2)^2 + (u_{xy}n_1 + u_{yy}n_2)^2 \\ 
        &= (u_{xx}^2+u_{xy}^2)n_1^2 + (u_{yx}^2+u_{yy}^2)n_2^2 + 2n_1n_2(u_{xx}u_{yx} + u_{xy}u_{yy}).
    \end{aligned}
\end{equation*}
For the cross-product term above, we invoke \Cref{lemma:youngs_ineq} with $p=q=2$  
\begin{equation*}
    \begin{aligned}
        2n_1n_2(u_{xx}u_{yx} + u_{xy}u_{yy}) &= 2\left( (n_2u_{xx})(n_1u_{yx})+(n_2u_{xy})(n_1u_{yy})\right)\\
        &\leq (u_{xx}^2+u_{xy}^2)n_2^2 + (u_{yx}^2+u_{yy}^2)n_1^2.
    \end{aligned}
\end{equation*}
Since $n_1^2+n_2^2=1$, we obtain
\begin{equation*}
    \left\|\nabla (\nabla u\cdot\mathbf{n}) \right\|^2 
    \leq u_{xx}^2 + u_{xy}^2 + u_{yx}^2 + u_{yy}^2.
\end{equation*}
Using this and the trace inequality gives
\begin{equation*}
    \begin{aligned}
        \left\|\nabla (\nabla u\cdot\mathbf{n})\right\|^2_{L^2(\partial K)} &= \int_{\partial K} \left\|\nabla (\nabla u\cdot\mathbf{n}) \right\|^2 ~ds  \\
        &\leq \int_{\partial K} \left(u_{xx}^2 + u_{xy}^2 + u_{yx}^2 + u_{yy}^2\right) ~ds \\
        &\leq C \frac{k^2}{h} \int_{K} \left(u_{xx}^2 + u_{xy}^2 + u_{yx}^2 + u_{yy}^2\right) ~dxdy 
        = C \frac{k^2}{h} \left| u \right|^2_{H^2(K)}.
    \end{aligned}
\end{equation*}
By \Cref{lemma:bound-on-sobolev-seminorm} with $d=l=p=r=2$ and $m=1$, we have
\begin{equation*}
    \left| u \right|_{H^2(K)} \leq \frac{C}{h} \left| u \right|_{H^1(K)}=\frac{C}{h} \left\| \nabla u \right\|_{L^2(K)},
\end{equation*}
which leads to
\begin{equation}\label{eqn:surf_jump_second_deriv_final}
    \left\|\nabla (\nabla u\cdot\mathbf{n})\right\|^2_{L^2(\partial K)} 
    \leq C \frac{k^2}{h^3} \left\| \nabla u \right\|^2_{L^2(K)}.
\end{equation}
Finally, substituting \Cref{eqn:surf_jump_second_deriv_equiv,eqn:surf_jump_second_deriv_final} in \Cref{eqn:surf_jump_second_deriv_intermediate_ineq} leads to the desired result.
\end{proof}
\bibliographystyle{siamplain}
\bibliography{references}
\end{document}